\documentclass[a4paper,12pt]{amsart}

\usepackage{t1enc}

\def\Bbb{\mathbb}
\def\Cal{\mathcal}

\def\R{\mathbb{R}}
\def\W{\mathbb{W}}
\def\X{\mathbb{X}}
\def\Y{\mathbb{Y}}
\def\Z{\mathbb{Z}}
\def\N{\mathbb{N}}

\def\cE{\mathcal{E}}
\def\cV{\mathcal{V}}
\def\cT{\mathcal{T}}

\def\cH{\mathcal{H}}
\def\cN{\mathcal{N}}

\def\cR{\mathcal{R}}

\def\al{\alpha}
\def\be{\beta}
\def\ga{\gamma}
\def\de{\delta}
\def\ep{\varepsilon}

\def\io{\iota}
\def\ka{\kappa}
\def\la{\lambda}
\def\rh{\rho}
\def\si{\sigma}

\def\ph{\varphi}

\def\ps{\psi}
\def\om{\omega}

\def\De{\Delta}

\def\Ph{\Phi}

\def\Om{\Omega}
\def\Up{\Upsilon}
\def\na{\nabla}

\def\form#1{\mathbf{#1}}
\def\dform#1{\dot{\mathbf{#1}}}
\def\ddform#1{\ddot{\mathbf{#1}}}

\newcommand{\id}{\operatorname{id}}
\newcommand{\End}{\operatorname{End}}

\newcommand{\lpl}{
  \mbox{$
  \begin{picture}(12.7,8)(-.5,-1)
  \put(2,0.2){$+$}
  \put(6.2,2.8){\oval(8,8)[l]}
  \end{picture}$}}
\renewcommand{\vec}[1]{\mathbf{#1}}

\newcommand{\wh}{\widehat}
\newcommand{\cq}{{\Cal Q}}

\newcommand{\modDe}{\fl}
\newcommand{\modD}{\fD}
\newcommand{\tmodD}{\widetilde{\fD}}
\newcommand{\modBox}  {\mbox{$                                    
  \begin{picture}(9,8)(1.6,0.15)
  \put(1,0.2){\mbox{$ \Box \hspace{-7.8pt} /$}}
\end{picture}$}} 

\newcommand{\tmodBox}{\widetilde{\modBox}}

\newcommand{\newc}{\newcommand}

\newtheorem{theorem}{Theorem}[section]
\newtheorem{lemma}[theorem]{Lemma}
\newtheorem{proposition}[theorem]{Proposition}

\newtheorem{definition}[theorem]{Definition}
\theoremstyle{remark}
\newtheorem{remark}[theorem]{\rm\bf Remark}
\newtheorem*{remark*}{\rm\bf Remark}

\newcommand{\ce}{{\Cal E}}

\usepackage{amssymb,stmaryrd}
\usepackage{amscd}

\newcommand{\nd}{\nabla}

\newcommand{\Rho}{P}

\newcommand{\nn}[1]{(\ref{#1})}

\newcommand{\D}{\mbox{\boldmath{$ D$}}}

\newcommand{\bg}{\mbox{\boldmath{$ g$}}}

\newcommand{\J}{J}

\newcommand{\ct}{{\Cal T}}

\newcommand{\fl}{\mbox{$
\begin{picture}(9,8)(1.6,0.15)
\put(1,0.2){\mbox{$ \Delta \hspace{-7.8pt} /$}}
\end{picture}$}}
\newcommand{\afl}{\mbox{$
\begin{picture}(9,8)(1.6,0.15)
\put(1,0.2){\mbox{$ {\mbox{\boldmath$\Delta$}} \hspace{-7.8pt} /$}}
\end{picture}$}}
\newcommand{\fD}{\mbox{$
\begin{picture}(9,8)(1.6,0.15)
\put(1,0.2){\mbox{$ D \hspace{-7.8pt} /$}}
\end{picture}$}}
\newcommand{\afD}
{\mbox{$
\begin{picture}(9,8)(1.6,0.15)
\put(1,0.2){\mbox{$ \D \hspace{-7.8pt} /$}}
\end{picture}$}}

\newc{\aR}{\mbox{\boldmath{$ R$}}}
\newc{\aS}{\mbox{\boldmath{$ S$}}}
\newc{\aDeR}{\mbox{\boldmath{$ U$}}_B{}^P{}_C{}^Q}
\newc{\aDe}{\mbox{\boldmath$ \Delta$}}
\newc{\aNd}{\mbox{\boldmath$ \nabla$}}

\newc{\aK}{\mbox{\boldmath{$ K$}}}
\newc{\aL}{\mbox{\boldmath{$ L$}}}
\newcommand{\h}{\mbox{\boldmath{$ h$}}}

\def\sideremark#1{\ifvmode\leavevmode\fi\vadjust{\vbox to0pt{\vss
 \hbox to 0pt{\hskip\hsize\hskip1em
 \vbox{\hsize3cm\tiny\raggedright\pretolerance10000
 \noindent #1\hfill}\hss}\vbox to8pt{\vfil}\vss}}}%
                        
\def\idx#1{{\em #1\/}}

\author{A. Rod Gover and Josef \v Silhan}
\email{gover@math.auckland.ac.nz} \title{Conformal operators on
  weighted forms; their decomposition and null space on Einstein manifolds}

\begin{document}

\begin{abstract}
There is a class of Laplacian like conformally invariant differential
operators on differential forms $L^\ell_k$ which may be considered as the
generalisation to differential forms of the conformally invariant
powers of the Laplacian known as the Paneitz and GJMS operators. On
conformally Einstein manifolds we give explicit formulae for these as
 factored polynomials in second order differential
operators. In the case the manifold is not Ricci flat we use this to
provide a direct sum decomposition of the null space of the $L^\ell_k$
in terms of the null spaces of mutually commuting second order
factors.
 \end{abstract}

\address{ARG: Department of Mathematics\\ The University of
  Auckland\\ Private Bag 92019\\ Auckland 1142\\ New
  Zealand; Mathematical Sciences Institute\\ Australian National
  University \\ ACT 0200, Australia} \email{r.gover@auckland.ac.nz}
\address{JS: Institute of Mathematics and Statistics \\ Masaryk
  University \\ Building 08 \\ Kotl\'a\v{r}sk\'a 2 \\ 611 37,
  Brno\\ Czech Republic} \email{silhan@math.muni.cz}

\maketitle

\pagestyle{myheadings}
\markboth{Gover \& \v Silhan}{Weighted forms on Einstein manifolds}

\section{Introduction}

On Riemannian and pseudo-Riemannian manifolds a natural differential
operator is said to be {\em conformally invariant} if it descends to a
well-defined differential operator on {\em conformal manifolds}, that
is manifolds equipped only with an equivalence class $c$ of metrics,
where $g,g'\in c$ means that $g'=fg$ for some positive function $f$.
This special class of operators have long held a central place in
mathematics and physics.  For example they govern the behaviour of
massless particles, and have desirable properties on conformally
compact manifolds, as for example used in general relativity
\cite{Dirac,MasonJNP}. They play a role in curvature prescription, in
extremal problems for metrics both in Riemannian geometry as well as
in string and brane theories
\cite{Branson-Orsted,CYAnnals,DM,GrahamWitten,schoen}. Via the
Fefferman bundle and metric, conformal structure and the corresponding
differential operators are also important in complex and CR geometry
\cite{FeffH,GoGr}.

A  rich programme surrounds the conformally invariant
differential operators $P_{2k}$ with leading term a power of the
Laplacian $\Delta^{k}$, see e.g.\ \cite{tomsharp,CGJP,FGbook,GrJ,GrZ,JuhlBook}
and references therein.  This family of operators on scalar functions
(or more accurately conformal densities) consists of the second order
{\em conformal Laplacian} \cite{Veblen}, (also called the {\em Yamabe
  operator}), the 4th order Paneitz operator \cite{Paneitz}, and
higher order {\em GJMS operators} of \cite{GJMS}.

Many of the important features of the GJMS family are shared by a
larger family of operators $L_k^\ell$ introduced in
\cite{BrGodeRham}. (Related operators of order 2, 4 and 6 were
constructed in \cite{tbms}. See also \cite{DesN} for the role of the
order 2 operator in Physics.) Each of these act on density-valued
differential forms ${\mathcal{E}}^k[w]$ (for the notation see Section
\ref{back}) and, up to a non-zero constant multiple, takes the form
\begin{equation*}
\underbrace{(n-2k+2\ell)}_{-2u}(\delta d)^\ell+
\underbrace{(n-2k-2\ell)}_{-2w}
(d\delta)^\ell+\mbox{ lower order terms},
\end{equation*}
 and carries ${\mathcal{E}}^k[w]$ to ${\mathcal{E}}^k[u]$.  Here $d$
 is the exterior derivative, $\delta$ its formal adjoint and this
 formula follows from the formulae for the operators on the sphere
 (\cite{tbjfa}, Remark 3.30).  At the $k=0$ specialisation (and with
 restrictions on $\ell$ in the case of even dimensions) these are the
 usual GJMS operators of \cite{GJMS}, see \cite{BrGodeRham}. For other
 $k$ and when neither $u$ or $w$ is zero they are evidently still {\em
   Laplacian like} in the sense that the composition $(w \delta d+ u d
 \delta )\circ L^k_\ell$ takes the form
$$
(w \delta d+ u d \delta )\circ L^k_\ell \sim \Delta^{\ell+1} 
+\mbox{ lower order terms} ,
$$ 
where the ``$\sim$'' means up to a non-zero constant multiple. This
means immediately that the operators $L^k_\ell$ are elliptic in the
case of Riemannian signature (conformal structures), or
hyperbolic/ultrahyperbolic in the case of other signatures.  If $u$ or
$w$ is zero (and $k\neq 0,n$) then these operators instead arise in
conformally invariant differential complexes (namely {\em detour
  complexes}) with the corresponding properties; for example the
complexes concerned are elliptic on Riemannian signature backgrounds.
In all cases the operators yield globally conformally invariant
pairings on compactly supported sections
\begin{equation}\label{pair}
\ce^k[\ell+k-\frac{n}{2}] \ni \varphi,\psi \mapsto \langle \varphi,\psi \rangle :=\int_M \varphi\cdot  L^\ell_k \psi\,
d\mu_{\mbox{\scriptsize\boldmath{$g$}}},
\end{equation}
where $\varphi\cdot L^\ell_k \psi \in{\mathcal{E}}[-n]$ denotes a
complete metric contraction between $\varphi $ and $L^\ell_k \psi$,
and $d\mu_{\mbox{\scriptsize\boldmath{$g$}}}$ is the conformal
measure. The operators $L^\ell_k$ are formally self-adjoint, and so
this pairing is symmetric.

It is shown in \cite{BrGodeRham} that the family $L^k_\ell$ leads to
operators which are analogues of Branson's Q-curvature (of
\cite{Branson-Orsted,tomsharp}) and a host new conformal
invariants. The work \cite{AG} of Aubry-Guillarmou shows that these
objects play a deep role in geometric scattering, and one that
generalises the corresponding results for the GJMS operators and
Q-curvature; see also \cite{GLW}. In these constructions, the null space
(or kernel) of the operators $L^k_\ell$ is important.

Except at the lowest orders, explicit general formulae are not
available for the $L^k_\ell$. For any particular operator a formula
may be obtained algorithmically via tractor calculus by the theory
developed in \cite{GoPetLap,GoPetobst}. However it is clear from the
special case of GJMS operators that the resulting operators, when
presented in the usual way, would be given by extremely complicated
formulae. 

It was shown by Graham \cite{Gr03, FGbook} using the Fefferman-Graham
ambient metric, and via a new construction in \cite{GoEinst}, that for
the GJMS operators striking simplifications are available on
conformally Einstein manifolds. In particular very simple
factorisation formulae are available which mean that, in this setting,
the operators may be given by a formula which is no more complicated
than Branson's corresponding formula for the case of the
standard round sphere \cite{tomsharp}. (We note here that similar
formulae are available for the conformal ``powers'' of the Dirac
operator on the sphere \cite{LR,ESou}.) This means that in this
setting (and without restriction on signature or compactness) one may
obtain explicit decompositions of the null space of the GJMS operators
in terms of eigenspaces of the Laplacian \cite{GoSiDec}.

The situation is considerably more subtle for the operators $L^k_\ell$ on 
weighted differential forms $\cE^k[w]$, with the situation varying significantly
between different cases. In \cite{GoSiEHarm} it was shown that
for the case of $w=0$ (i.e.\ when the domain space is that of {\em
  true} or {\em unweighted} differential forms) it is possible to
treat the operators $L^k_\ell$ in a manner very similar to the
treatment of the GJMS operators in \cite{GoEinst}. The result was that, 
 in setting of an Einstein structure, one may recover rather
directly the detour complexes and Q-operators on differential
forms. What is more these were shown to be given by simple explicit
formulae; these involved factorisations of the operators which
generalise those from \cite{GoEinst}.

The aim of this article is to complete the picture by treating the
operators $L^k_\ell$ when $w\neq 0$, again in the setting of an
Einstein structure. This case is much more difficult. On true forms,
as treated in \cite{GoSiEHarm}, the exterior derivative $d$ acts in a
conformally invariant way and the operators $L^k_\ell$ factor though
this, and its conformal adjoint $\delta$ (which is also conformally
invariant in the way it arises). Indeed on true forms the operators
$L_k^\ell$, have $\ell = n/2-k$ (which may thus be omitted in the
notation) and take the form $L_k=\delta Q_{k+1} d: \cE^k[0] \to
\cE^k[-n+2k]$ for suitable operators $Q_{k+1}$. But the form of this
composition immediately implies simplifications of the middle
operators $Q_{k+1}$, by dint of the identities $d\circ d=0$ and
$\delta \circ \delta =0$. Similar observations also simplify the
tractor formulae involved in that case. However these simplifications
are not available for the remaining weights $w\neq 0$, and thus
significant new ideas were required here.  The key tool, that we
develop and use, is Theorem \ref{MM*}. This gives a basic
factorisation result that can then be used to inductively derive our
final results.  In fact this approach enables us to deal with all
weights; thus we recover and significantly extend many of the results
of \cite{GoSiEHarm}.

We show that the operators $L_k^\ell: \cE^k[w] \to \cE^k[w-2\ell]$, 
where $w=k+\ell-n/2$ 
and $\ell \geq 1$,  may be expressed as compositions $L_k^\ell =
S_1 \ldots S_\ell$ where every factor on the right hand side has the
form $S_i = ad\de + b\de d+c$ for some scalars $a,b,c \in \mathbb{R}$.
(Note these 2nd order factors commute.)  
In detail, we have
the following theorem which completely describes the differential operators
$L^\ell_k$ on Einstein manifolds.
\begin{theorem}\label{main}
Let $\Ph := \{1,\ldots,\ell\}$ and $w=k+\ell-n/2$ where $1 \leq k \leq
\frac{n}{2}$.  The operator
$$ 
L_k^\ell: \cE^k[w] \to \cE^k[w-2\ell]
$$
has the explicit form
$$
L_k^\ell \sim \begin{cases}
(d\de-\de d) P_k^{\Ph \setminus \{\ell\}}(d\de,\de d) & k=n/2 \\
P_k^{\Ph}(d\de,\de d) & w \leq 0 \quad \mbox{and} \quad k<n/2 \\
\widetilde{P}_k(d\de,\de d)P_k^{\Ph \setminus \{w,w+1\}}(d\de,\de d)
& w \geq 1 \quad \mbox{and} \quad k<n/2
\end{cases},
$$
for $n$ even and 
$$
L_k^\ell \sim P_k^{\Ph}(d\de,\de d)
$$
for $n$ odd where $\sim$ means ``is equal up to nonzero scalar multiple''
and two variable polynomials $P_k$ and $\widetilde{P}_k$ are given in
\nn{Yam} and \nn{SqYam}.
\end{theorem}
\noindent As mentioned earlier, Theorem \ref{MM*} provides the main step
needed to reach the Theorem above. Note that Theorem \ref{MM*}  is interesting
and important in its own right.

In \cite{GoSiDec} we show that if the factors $S_i:\cV\to \cV$ (for
some vector space $\cV$), in a composition $P:=S_1 \cdots S_\ell$ of
mutually commuting operators, are suitably {\em relatively
  invertible}, then the general inhomogeneous problem $Pu=f$
decomposes into an equivalent system $S_iu_i=f$, $i=1,\cdots , \ell$.
For the factors of the operators $L_k^\ell$ a sufficient form of
relative invertibility is established in Proposition \ref{identity},
in the case that the Einstein manifold is not Ricci flat.  This is
then used to reduce the generally high order conformal operators
$L_k^\ell$ to equivalent lower order systems. The outcome is that in
any signature (and without any assumption of compactness) on non-Ricci
flat Einstein manifolds we can describe the spaces $\cN(L_k^\ell)$
(the null space of $L_k^\ell$), explicitly as a direct sum of null
spaces for the second order factors of the $L_k^\ell$ (as in the
Theorem above). This is Theorem \ref{LdecGen}.  In the case of
Riemannian signature and compact manifolds the situation is vastly
simpler as the Hodge decomposition may be used in conjunction with the
Theorem \ref{main}, and so we obtain Theorem \ref{LdecSpec}.

In the Einstein setting we give a direct definition of the operators
$L_k^\ell$ in the tractor calculus, see Definition \nn{Ldef} of
Section \ref{four}. In that Section we explain the consistency of our
definition with that in \cite{BrGodeRham}.

ARG gratefully acknowledges support from the Royal
Society of New Zealand via Marsden Grant 10-UOA-113;
J\v{S} was supported by the grant agency of the Czech republic under the grant
P201/12/G028.

\section{Background: Einstein metrics and conformal geometry} \label{back}

Let $M$ be a smooth manifold, equipped with a Riemannian metric
$g_{ab}$. Here and throughout we employ Penrose's abstract index
notation. We shall write $\ce^a$ to denote the space of smooth
sections of the tangent bundle $TM$ on $M$, and $\ce_a$ for the space
of smooth sections of the cotangent bundle $T^*M$.  (In fact we will
often use the same symbols for the bundles themselves.) We write $\ce$
for the space of smooth functions and all tensors considered will be
assumed smooth without further comment.  An index which appears twice,
once raised and once lowered, indicates a contraction.  
For simplicity we shall assume that the
manifold $M$ has dimension $n\geq 3$.

We first sketch here notation and background for general conformal
structures and their tractor calculus following
\cite{CapGoamb,GoPetLap}.  Recall that a {\em conformal structure\/}
of signature $(p,q)$ on $M$ is an equivalence class $c$ of metrics,
where the equivalence relation $g\sim \widehat{g}$ of metrics in $c$
is that $\widehat{g} = f g$ for some positive function
$f$. Equivalently a conformal structure is a smooth ray subbundle
$\cq\subset S^2T^*M$ whose fibre over $x$ consists of conformally
related signature-$(p,q)$ metrics at the point $x$. Sections of $\cq$
are metrics $g$ on $M$. The principal bundle $\pi:\cq\to M$ has
structure group $\Bbb R_+$, and each representation ${\Bbb R}_+ \ni
x\mapsto x^{-w/2}\in {\rm End}(\Bbb R)$ induces a natural line bundle
on $ (M,[g])$ that we term the conformal density bundle $E[w]$. We
shall write $ \ce[w]$ for the space of sections of this bundle.  Given
a vector bundle $V$ (or its space of sections $\mathcal{V}$) we shall
write $V[w]$ (resp.\ $\mathcal{V}[w]$) to mean $V\otimes E[w]$
(resp.\ $\mathcal{V}\otimes \ce[w]$).  Here and throughout, sections,
tensors, and functions are always smooth, meaning $C^\infty$.  When no
confusion is likely to arise, we will use the same notation for a
bundle and its section space.

We write $\bg$ for the {\em conformal metric}, that is the
tautological section of $S^2T^*M\otimes E[2]$ determined by the
conformal structure. This will be used to identify $TM$ with
$T^*M[2]$.  For many calculations we will use abstract indices in an
obvious way.  Given a choice of metric $ g$ from the conformal class,
we write $ \nabla$ for the corresponding Levi-Civita connection. With
these conventions the Laplacian $ \Delta$ is given by
$\Delta=\bg^{ab}\nd_a\nd_b= \nd^b\nd_b\,$.  Note $E[w]$ is trivialised
by a choice of metric $g$ from the conformal class, and we write $\nd$
for the connection arising from this trivialisation.  It follows
immediately that (the coupled) $ \nd_a$ preserves the conformal
metric.

Since the Levi-Civita connection is torsion-free, the (Riemannian)
curvature 
$R_{ab}{}^c{}_d$ is given by $ [\nd_a,\nd_b]v^c=R_{ab}{}^c{}_dv^d $ where 
$[\cdot,\cdot]$ indicates the commutator bracket.  The Riemannian
curvature can be decomposed into the totally trace-free Weyl curvature
$C_{abcd}$ and a remaining part described by the symmetric {\em
Schouten tensor} $\Rho_{ab}$, according to $
R_{abcd}=C_{abcd}+2\bg_{c[a}\Rho_{b]d}+2\bg_{d[b}\Rho_{a]c}, $ where
$[\cdots]$ indicates antisymmetrisation over the enclosed indices.
We put $J := P^a{}_a$. The {\em
Cotton tensor} is defined by
$$
A_{abc}:=2\nabla_{[b}\Rho_{c]a} .
$$
Under a {\em conformal transformation} we replace a choice of metric $
g$ by the metric $ \hat{g}=e^{2\om} g$, where $\omega$ is a smooth
function. Explicit formulae for the corresponding transformation of
the Levi-Civita connection and its curvatures are given in e.g.\ 
\cite{GoPetLap}. We recall that, in particular, the Weyl curvature is
conformally invariant $\widehat{C}_{abcd}=C_{abcd}$.

\subsection{Conformally invariant operators on weighted forms}

Following \cite{BrGodeRham} we shall write $\ce^k$ to denote the
section space of smooth $k$-forms and $\ce^k[w]=\ce^k\otimes \ce[w]$.
Although this is similar to the notation for (weighted) tangent
sections, by context no confusion should arise.

The invariant differential operators on conformally flat manifolds
are all known. We shall refer to the summary of the classification in 
\cite[section 3]{ES}. In particular, on weighted $k$-forms $\cE^k[w]$ 
 there 
are three types of operators. 
Here we shall  mainly focus on the (power) Laplacian like operators 
$$
L_k^\ell: \cE^k[w] \to \cE^k[w-2\ell], \qquad 
k \leq \frac{n}{2},\ w = k+\ell-n/2 .
$$
Here $n$ is the dimension and $\ell \geq 1$ the order, i.e.\
$w \geq k-\frac{n}{2}+1$. That is, $w$ is an integer for $n$ even and 
a half integer for $n$ odd. Using the terminology of \cite[section 3]{ES}, 
these operators are all non--standard for $n$ odd and, assuming $n$ even, 
they are regular for $w \geq k+1$ and $w=0$ and singular in remaining cases, 
i.e.\ for $w \in \{k-\frac{n}{2}+1, \ldots, k\} \setminus \{0\}$. Further 
possible invariant operators are the exterior derivative $d$ and 
its formal adjoint $\de$
$$
d:\cE^k[0] \to \cE^{k+1}[0],\ k \leq n-1 \quad \text{and} \quad 
\de: \cE^k[-n+2k] \to \cE^{k-1}[-n+2k-2],\ k \geq 1
$$ 
of differential order 1. We extend the use of this notation to
weighted differential forms in the obvious way via the Levi-Civita
connection; for $f\in \ce^k[w]$ we write $df $ and $\delta f$ to mean
$$
(k+1)\nabla_{[a_0} f_{a_1\cdots a_k]} \quad \mbox{and} \quad  - \nabla^{a_1} f_{a_1\cdots a_k},
$$ 
respectively.

Finally, for integers $w \geq k+1$, $k \geq 1$ and $w' \geq 1$ there are 
overdetermined operators 
$$
\cE^k[w] \to \underbrace{\cE^1 \otimes \ldots \otimes \cE^1}_{w-k}
\otimes \cE^k[w] 
\quad \mbox{and} \quad
\cE[w'] \to \underbrace{\cE^1 \otimes \ldots \otimes \cE^1}_{w'+1}[w']
$$ 
of differential order $w-k$ and $w'+1$, respectively. More
precisely, the target bundle is the Cartan component of the displayed
space. These operators are regular and are a class of what are known
as first BGG operators.

\subsection{Conformal geometry and tractor calculus}\label{tractorsect}

A central tool in the treatment of conformal geometry is tractor
calculus \cite{BEGo}, since this is a conformally invariant replacement
of the Ricci calculus of pseudo-Riemannian geometry. (For a general
development of tractor calculus in the broader context of all parabolic
geometries see \cite{CapGotrans}). The discussion here follows
\cite{GoPetLap} and for the treatment of forms
\cite{BrGodeRham,GoSiKil,Sthesis} as summarised in \cite{GoSiEHarm}.
Some parts of the treatment are specialised to Einstein manifolds.

We first recall the definition of the standard tractor bundle over
$(M,[g])$.  This is a vector bundle of rank $n+2$ defined, for each
$g\in[g]$, by $[\ce^A]_g=\ce[1]\oplus\ce_a[1]\oplus\ce[-1]$.  If $\wh
g=e^{2\Up}g$, we identify $(\alpha,\mu_a,\tau)\in[\ce^A]_g$ with
$(\wh\alpha,\wh\mu_a,\wh\tau)\in[\ce^A]_{\wh g}$ by the transformation
\begin{equation}\label{transf-tractor}
 \begin{pmatrix}
 \wh\alpha\\ \wh\mu_a\\ \wh\tau
 \end{pmatrix}=
 \begin{pmatrix}
 1 & 0& 0\\
 \Up_a&\delta_a{}^b&0\\
- \tfrac{1}{2}\Up_c\Up^c &-\Up^b& 1
 \end{pmatrix} 
 \begin{pmatrix}
 \alpha\\ \mu_b\\ \tau
 \end{pmatrix} ,
\end{equation}
where $\Up_a:=\nd_a \Up$.  It is straightforward to verify that these
identifications are consistent upon changing to a third metric from
the conformal class, and so taking the quotient by this equivalence
relation defines the {\em standard tractor bundle} $\ct$, or $\ce^A$
in an abstract index notation, over the conformal manifold.
(Alternatively the standard tractor bundle may be constructed as a
canonical quotient of a certain 2-jet bundle or as an associated
bundle to the normal conformal Cartan bundle \cite{luminy}.) On a
conformal structure of signature $(p,q)$, the bundle $\ce^A$ admits an
invariant metric $ h_{AB}$ of signature $(p+1,q+1)$ and an invariant
connection, which we shall also denote by $\nabla_a$, preserving
$h_{AB}$.  In a conformal scale $g$, these are given by
\begin{equation}\label{basictrf}
 h_{AB}=\begin{pmatrix}
 0 & 0& 1\\
 0&\bg_{ab}&0\\
1 & 0 & 0
 \end{pmatrix}
\text{ and }
\nabla_a\begin{pmatrix}
 \alpha\\ \mu_b\\ \tau
 \end{pmatrix}
 =
\begin{pmatrix}
 \nabla_a \alpha-\mu_a \\
 \nabla_a \mu_b+ \bg_{ab} \tau +\Rho_{ab}\alpha \\
 \nabla_a \tau - \Rho_{ab}\mu^b  \end{pmatrix}. 
\end{equation}
It is readily verified that both of these are conformally well-defined,
i.e., independent of the choice of a metric $g\in [g]$.  Note that
$h_{AB}$ defines a section of $\ce_{AB}=\ce_A\otimes\ce_B$, where
$\ce_A$ is the dual bundle of $\ce^A$. Hence we may use $h_{AB}$ and
its inverse $h^{AB}$ to raise or lower indices of $\ce_A$, $\ce^A$ and
their tensor products.

In computations, it is often useful to introduce 
the `projectors' from $\ce^A$ to
the components $\ce[1]$, $\ce_a[1]$ and $\ce[-1]$ which are determined
by a choice of scale.
They are respectively denoted by $X_A\in\ce_A[1]$, 
$Z_{Aa}\in\ce_{Aa}[1]$ and $Y_A\in\ce_A[-1]$, where
 $\ce_{Aa}[w]=\ce_A\otimes\ce_a\otimes\ce[w]$, etc.
 Using the metrics $h_{AB}$ and $\bg_{ab}$ to raise indices,
we define $X^A, Z^{Aa}, Y^A$. Then we
immediately see that 
$$
Y_AX^A=1,\ \ Z_{Ab}Z^A{}_c=\bg_{bc} ,
$$
and that all other quadratic combinations that contract the tractor
index vanish. 
In \eqref{transf-tractor} note that  
$\wh{\alpha}=\alpha$ and hence $X^A$ is conformally invariant. 

Given a choice of conformal scale, the {\em tractor-$D$ operator} 
$D_A\colon\ce_{B \cdots E}[w] \to \cE_{AB\cdots E}[w-1]$
is defined by 
\begin{equation}\label{tractorD}
D_A V:=(n+2w-2)w Y_A V+ (n+2w-2)Z_{Aa}\nabla^a V -X_A\Box V, 
\end{equation} 
where $\Box V :=\Delta V+w \J V$.  This also turns out to be
conformally invariant as can be checked directly using the formulae
above (or alternatively there are conformally invariant constructions
of $D$, see e.g.\ \cite{Gosrni}).

The curvature $ \Omega$ of the tractor connection 
is defined by 
$$
[\nd_a,\nd_b] V^C= \Omega_{ab}{}^C{}_EV^E 
$$
for $ V^C\in\ce^C$.  Using
\eqref{basictrf} and the formulae for the Riemannian curvature yields
\begin{equation}\label{tractcurv}
\Omega_{abCE}= Z_C{}^cZ_E{}^e C_{abce}-2X_{[C}Z_{E]}{}^e A_{eab}
\end{equation}

We will also need a conformally invariant curvature quantity defined as
follows (cf.\ \cite{Gosrni,Goadv})
\begin{equation}\label{Wdef}
W_{BC}{}^E{}_F:=
\frac{3}{n-2}D^AX_{[A} \Omega_{BC]}{}^E{}_F ,
\end{equation}
where $\Omega_{BC}{}^E{}_F:= Z_B{}^bZ_C{}^c \Om_{bc}{}^E{}_F$.
In a choice of conformal scale, 
 $W_{ABCE}$ is given by
\begin{equation}\label{Wform}
\begin{array}{l}
(n-4)\left( Z_A{}^aZ_B{}^bZ_C{}^cZ_E{}^e C_{abce}
-2 Z_A{}^aZ_B{}^bX_{[C}Z_{E]}{}^e A_{eab}\right. \\ 
\left.-2 X_{[A}Z_{B]}{}^b Z_C{}^cZ_E{}^e A_{bce} \right)
+ 4 X_{[A}Z_{B]}{}^b X_{[C} Z_{E]}{}^e B_{eb},
\end{array}
\end{equation}
where 
$$
B_{ab}:=\nabla^c
A_{acb}+\Rho^{dc}C_{dacb}.
$$ is known as the {\em Bach tensor}. From the formula \nn{Wform} it
is clear that $W_{ABCD}$ has Weyl tensor type symmetries.

We will work with conformally Einstein manifolds. That is, conformal
structures with an Einstein metric in the conformal class. This is the
same as the existence of a non-vanishing section $\si \in
\mathcal{E}[1]$ satisfying $\left[ \na_{(a}\na_{b)_0} + P_{(ab)_0}
  \right] \si =0$ where the subscript $(\ldots)_0$ indicates the
trace-free symmetric part.  Equivalently (see e.g.\ \cite{BEGo,GoNur})
there is a standard tractor $I_A$ that is parallel with respect to the
normal tractor connection $\na$ and such that $\si:=X_AI^A$ is
non-vanishing. It follows that $I_A := \frac{1}{n}D_A \si = Y_A \si
+Z_A^a \na_a \si -\frac{1}{n} X_A (\De+J) \si$, for some section
$\si\in \ce[1]$, and so $X^AI_A=\si$ is non-vanishing. If we compute
in the scale $\si$, then the Cotton and Bach tensors are zero (see
e.g.\ \cite{GoNur}) and so $W_{ABCD} = (n-4) Z_A^a Z_B^b Z_C^c Z_D^d
C_{abcd}$.

\subsection{Tractor forms} Here we recall the calculus for tractor forms 
as developed in  \cite{GoSiEHarm}. We
write $\cE^k[w]$ for the space of sections of $(\Lambda^k T^*M)
\otimes E[w]$ (and $\ce^k= \ce^k[0]$).  Further we put $\cE_k[w] :=
\cE^k[w+2k-n]$. 
We shall use the analogous notation $\cT^k[w]
:= (\Lambda^k \cT) \otimes \cE[w]$ on the tractor level.

In order to be
explicit and efficient in calculations involving bundles of possibly
high rank it is necessary to employ abstract index notation
as follows.  
In the usual abstract index conventions one would write
$\ce_{[ab\cdots c]}$ (where there are implicitly $k$-indices skewed
over) for the space $\ce^k$. To simplify subsequent expressions we
 use the following conventions. Firstly
indices labelled with sequential superscripts which are
at the same level (i.e.\ all contravariant or all
covariant) will indicate a completely skew set of indices.
Formally we set $a^1 \cdots a^k = [a^1 \cdots a^k]$ and so, for example,
$\ce_{a^1 \cdots a^k}$ is an alternative notation for $\ce^k$
while $\ce_{a^1 \cdots a^{k-1}}$ and $\ce_{a^2 \cdots a^k}$ both denote
$\ce^{k-1}$. Next, following \cite{GoSiKil} we
abbreviate this notation via multi-indices: We will use the forms
indices
$$
\begin{aligned}
 \vec{a}^k &:=a^1 \cdots a^k =[a^1 \cdots a^k], \quad k \geq 0,\\
\dot{\vec{a}}^k &:= a^2 \cdots a^k=[a^2 \cdots a^k], \quad k \geq 1.
\end{aligned}
$$
If $k=1$ then $\dform{a}^k$ simply means the index is absent.
The corresponding notations will be used for tractor indices so
e.g. the bundle of tractor $k$--forms $\ce_{[A^1\cdots A^k]}$ will be
denoted by $\ce_{A^1\cdots A^k}$ or $\mathcal{E}_{\vec{A}^k}$.

The structure of $\mathcal{E}_{\vec{A}^k}$ is
\begin{equation} \label{comp_series_form}
  \mathcal{E}_{[A^1 \cdots A^k]} = \mathcal{E}_{\vec{A}^k} \simeq
  \mathcal{E}^{k-1}[k] \lpl \left( \mathcal{E}^k[k] \oplus
  \mathcal{E}^{k-2}[k-2] \right) \lpl \mathcal{E}^{k-1}[k-2];
\end{equation}
in a choice of scale the semidirect sums $\lpl$ may be replaced by
direct sums and otherwise they indicate the composition series
structure arising from the tensor powers of \nn{transf-tractor}.

In a choice of metric $g$ from the conformal class, the
projectors (or splitting operators) $X,Y,Z$ for $\mathcal{E}_A$
determine corresponding projectors $\X,\Y,\Z,\W$ for
$\mathcal{E}_{\vec{A}^{k+1}}$, $k \geq 1$ 
These execute the  splitting of this space into four components and are given 
as follows.
\begin{center}
\renewcommand{\arraystretch}{1.3}
\begin{tabular}{c@{\;=\;}l@{\;=\;}l@{\;=\;}l@{\ $\in$\ }l}
$\Y^k$ & $\Y_{A^0A^1 \cdots A^k}^{\quad a^1 \cdots\, a^k}$ &
  $\Y_{A^0\vec{A}^k}^{\quad \vec{a}^k}$ & $Y_{A^0}^{}Z_{A^1}^{a^1} \cdots Z_{A^k}^{a^k}$ &
  $\mathcal{E}_{\vec{A}^{k+1}}^{\vec{a}^k}[-k-1]$ \\
$\Z^k$ & $\Z_{A^1 \cdots A^k}^{\, a^1 \cdots\, a^k}$ &
  $\Z_{\vec{A}^k}^{\,\vec{a}^k}$ & $Z_{A^1}^{\,a^1} \cdots Z_{A^k}^{\,a^k}$ &
  $\mathcal{E}_{\vec{A}^k}^{\vec{a}^k}[-k]$ \\
$\W^k$ & $\W_{A'A^0A^1 \cdots A^k}^{\quad\,\ \ a^1 \cdots\, a^k}$ &
  $\W_{A'A^0\vec{A}^k}^{\quad\,\ \ \vec{a}^k}$ &
  $X_{[A'}^{}Y_{A^0}^{}Z_{A^1}^{\,a^1} \cdots Z_{A^k]}^{\,a^k}$ &
  $\mathcal{E}_{\vec{A}^{k+2}}^{\vec{a}^k}[-k]$ \\
$\X^k$ & $\X_{A^0A^1 \cdots A^k}^{\quad a^1 \cdots\, a^k}$ &
  $\X_{A^0\vec{A}^k}^{\quad \vec{a}^k}$ & $X_{A^0}^{}Z_{A^1}^{\,a^1} \cdots Z_{\,A^k}^{a^k}$ &  
  $\mathcal{E}_{\vec{A}^{k+1}}^{\vec{a}^k}[-k+1]$
\end{tabular}
\end{center}
where $k \geq 0$. The superscript $k$ in $\Y^k$, $\Z^k$, $\W^k$
and $\X^k$ shows the corresponding tensor valence. (This is
slightly different than in \cite{BrGodeRham}, where $k$ is
the relevant tractor valence.) Note that $Y=\Y^0$, $Z=\Z^1$ and $X=\X^0$ and
$\W^0 = X_{[A'}Y_{A^0]}$. To simplify notation we introduce 
projectors/injectors
\begin{align*}
&q^k: \cT^k[w] \to \cE^k[w+k], \quad F_\form{A}^{} \mapsto 
\Z^\form{A}_{\,\form{a}} F_\form{A} \ \ \text{for}\ \ F \in \cT^k \quad 
\text{and} \\
&q_k: \,\cE^k[w] \to \cT^k[w-k], \quad f_\form{a}^{} \mapsto 
\Z^{\,\form{a}}_{\form{A}} f_\form{a} \quad \ \text{for}\ \ f \in \cE^k.
\end{align*}

From \nn{basictrf} we immediately see $\na_p Y_A = Z_A^a P_{pa}$, 
$\na_p Z_A^a = -\delta_p^a Y_A - P_p^a X_A$ and $\na_p X_A = Z_{Ap}$.
From this we obtain 
the formulae (cf.\ \cite{GoSiKil}) 
\begin{equation} \label{na}
\begin{split}
  \na_p \Y_{A^0\form{A}^k}^{\quad \form{a}^k} &=
      P_{pa_0} \Z_{A^0\form{A}^k}^{\,a^0\,\form{a}^k}
    + k P_p^{\ a^1} \W_{A^0\form{A}^k}^{\quad \dot{\form{a}}^k} \\
  \na_p \Z_{A^0\form{A}^k}^{\,a^0\,\form{a}^k} &=
    - (k+1) \de_p^{a^0} \Y_{A^0\form{A}^k}^{\quad \form{a}^k}
    - (k+1) P_p^{\ a^0}  \X_{A^0\form{A}^k}^{\quad \form{a}^k} \\
  \na_p \W_{A^0\form{A}^k}^{\quad\, \dot{\form{a}}^k} &=
    - \bg_{pa^1} \Y_{A^0\form{A}^k}^{\quad \form{a}^k}
    + P_{pa^1} \X_{A^0\form{A}^k}^{\,\ \ a^1\!\dot{\form{a}}^k} \\
  \na_p \X_{A^0\form{A}^k}^{\quad \form{a}^k} &=
      \bg_{pa^0} \Z_{A^0\form{A}^k}^{\,a^0\,\form{a}^k}
    - k \de_p^{a^1} \W_{A^0\form{A}^k}^{\quad \dot{\form{a}}^k},
\end{split}
\end{equation}
which determine the tractor connection on form tractors in a conformal
scale. Similarly, one can compute the Laplacian $\De$ applied to the
tractors $\X$, $\Y$, $\Z$ and $\W$.  As an operator on form tractors
we have the opportunity to modify $\De$ by adding some amount of
$W\sharp\sharp$, where $\sharp $ denotes the natural tensorial
action of sections in $\End(\ce^A)$. Analogously, we shall use
$C \sharp\sharp$ to modify the Laplacian on forms; here $\sharp$ 
denotes the natural tensorial action of sections in $\End(\ce^a)$.
It turns out
(cf.\ \cite{BrGodeRham}) that it will be convenient for us to use modifications
\begin{equation} \label{modgen} 
\modDe = \De + \frac{1}{n-4} W\sharp\sharp \quad 
\text{and} \quad \modD_A = D_A - \frac{1}{n-4} X_A W\sharp\sharp \qquad
\text{for} \quad
n \not= 4,
\end{equation}
cf.\ \nn{tractorD}. (Note $\De = \na^a \na_a$.) The operator
$\modD$ was introduced in \cite{BrGodeRham}. 

Since the Laplacian is of the second order, it is  convenient 
to consider e.g.\ 
$\modDe \Y_{\form{A}}^{\,\dform{a}} \tau_{\dform{a}}$ where 
$\tau_{\dform{a}} \in \cE_{\dform{a}}[w]$. It will be sufficient for our 
purpose
to calculate this only in an Einstein scale. 
For example, using \nn{na} and then that $P_{ab}=\bg_{ab}J/n$, we have 
\begin{align*}
\na^p \na_p \Y_{\form{A}}^{\,\dform{a}} \tau_{\dform{a}} 
   =& \na^p \bigl[ 
    P_{pa^1} \Z_{\form{A}}^{\form{a}}
    + (k-1) P_p^{\ a^2} \W_{\form{A}}^{\ddot{\form{a}}}
    + \Y_{\form{A}}^{\,\dform{a}} \na_p  \bigr]\tau_{\dform{a}} \\
   =&-\Y_{\form{A}}^{\,\dform{a}}
     \bigl[ \bigl( \de d + d\de + (1\!-\!\frac{2(k\!-\!1)(n\!-\!k\!+\!1)}{n}) J
             +C \sharp\sharp \bigr)\tau \bigr]_{\dform{a}}   \\
    &+\frac{2}{nk} \Z_{\form{A}}^{\,\form{a}} (Jd\tau)_{\form{a}}
     -\frac{2(k\!-\!1)}{n} \W_{\form{A}}^{\,\ddform{a}} (J\de\tau)_{\ddform{a}} 
     -\frac{n\!-\!2k\!+\!2}{n^2} \X_{\form{A}}^{\,\dform{a}}
       J^2\tau_{\dform{a}}, 
\end{align*}
where, as usual, $\form{A}=\form{A}^k$ and $\form{a}=\form{a}^k$.
Summarising, one can compute that in
an Einstein scale we obtain 
\begin{equation} \label{dde-W}
\begin{split}
  - \modDe \Y_{\form{A}}^{\,\dform{a}} \tau_{\dform{a}} &= 
     \Y_{\form{A}}^{\,\dform{a}}
     \bigl[ \bigl( \de d + d\de + (1\!-\!\frac{2(k\!-\!1)(n\!-\!k\!+\!1)}{n}) J
            \bigr)\tau \bigr]_{\dform{a}}   \\
    &-\frac{2}{nk} \Z_{\form{A}}^{\,\form{a}} (Jd\tau)_{\form{a}}
     +\frac{2(k\!-\!1)}{n} \W_{\form{A}}^{\,\ddform{a}} (J\de\tau)_{\ddform{a}} 
     +\frac{n\!-\!2k\!+\!2}{n^2} \X_{\form{A}}^{\,\dform{a}}
       J^2\tau_{\dform{a}} \\
  - \modDe \Z_{\form{A}}^{\,\form{a}} \mu_{\form{a}} &= 
     -2k \Y_{\form{A}}^{\,\dform{a}} (\de\mu)_{\dform{a}} 
     +\Z_{\form{A}}^{\,\form{a}} \bigl[ \bigl(
       \de d + d\de - \frac{2k(n\!-\!k\!-\!1)}{n} J
       \bigr)\mu \bigr]_{\form{a}} \\
    &-\frac{2k}{n} \X_{\form{A}}^{\,\dform{a}} (J\de\mu)_{\dform{a}} \\
  - \modDe \W_{\form{A}}^{\,\ddform{a}} \nu_{\ddform{a}} &= 
       \frac{2}{k\!-\!1} \Y_{\form{A}}^{\,\dform{a}} (d\nu)_{\dform{a}} 
    +\W_{\form{A}}^{\,\ddform{a}} \bigl[ \bigl(
      \de d + d\de - \frac{2(k\!-\!3)(n\!-\!k\!+\!2)}{n}  J
       \bigr) \nu \bigr]_{\ddform{a}} \\
    &-\frac{2}{n(k\!-\!1)} \X_{\form{A}}^{\,\dform{a}} (d\nu)_{\dform{a}} \\
  - \modDe \X_{\form{A}}^{\,\dform{a}} \rh_{\dform{a}} &= 
     (n\!-\!2k\!+\!2) \Y_{\form{A}}^{\,\dform{a}} \rh_{\dform{a}} 
     -2(k\!-\!1) \W_{\form{A}}^{\,\ddform{a}} (\de\rh)_{\ddform{a}} \\ 
    &-\frac{2}{k} \Z_{\form{A}}^{\,\form{a}} (d\rh)_{\form{a}} 
     +\X_{\form{A}}^{\,\dform{a}} \bigl[ \bigl(
       \de d + d\de + (1\!-\!\frac{2(k\!-\!1)(n\!-\!k\!+\!1)}{n}) J
       \bigr) \rh \bigr]_{\dform{a}}.
\end{split}
\end{equation}
cf.\ \cite[(11)]{GoSiEHarm}.  Here $\tau_{\dform{a}}
\in \cE_{\dform{a}}[w]$, $\mu_{\form{a}} \in \cE_{\form{a}}[w]$,
$\nu_{\ddform{a}} \in \cE_{\ddform{a}}[w]$ and $\rh_{\dform{a}} \in
\cE_{\dform{a}}[w]$ where $\form{a} = \form{a}^k$, $k \geq 1$ and $w$
is any conformal weight. Note $\modDe$ is defined in \nn{modgen} for 
dimensions $n \not=4$ and under this assumption, \nn{dde-W} was stated in 
\cite{GoSiEHarm}. Below we extend the operator $\modDe$ also to the
dimension $n=4$, see \nn{modeinst}, and a short computation verifies that
\nn{dde-W} holds in the dimension $n=4$ as well. Summarising, 
\nn{dde-W} holds for all dimensions $n \geq 3$.

\subsection{Modified tractor $D$-operator on conformally Einstein manifolds}\label{mod}
The operator $\modD$ will be essential for our subsequent computation. It
is defined above only for $n \not=4$. Assuming $M$ is a conformally Einstein 
manifold, we extend this operator to the dimension $n=4$ as follows.
First we define the tractor
\begin{align} \label{Wform4}
\begin{split}
\widetilde{W}^\si_{ABCD} =  &\Z_{AB}^{\;a\;b} \Z_{CD}^{\;c\;d} C_{abcd}
-4 \Z_{AB}^{\;a\;b} \X_{CD}^{\ \ d} \na_{[a}P_{b]d} \\
&-4 \X_{AB}^{\ \ b} \Z_{CD}^{\;c\;d} \na_{[c}P_{d]b}
-8 \X_{AB}^{\ \ b} \X_{CD}^{\ \ d} \si^{-1} (\na_{[b}P_{a]d}) \na^a\si
\end{split}
\end{align}
in the scale $\si \in \cE[1]$. This (scale dependent quantity) was
used in \cite[section 4]{GoEinst} and it is shown there that
$\widetilde{W}^{\si_1}_{ABCD} = \widetilde{W}^{\si_2}_{ABCD}$ for any
pair of Einstein scales $\si_1$ and $\si_2$. Hence we can drop the
superscript $\si$ on conformally Einstein manifolds as
$\widetilde{W}_{ABCD}$ is well defined on such structures. 

Using $\widetilde{W}_{ABCD}$, we define new modifications
\begin{equation} \label{modeinst} 
\modDe = \De + \widetilde{W}\sharp\sharp \quad 
\text{and} \quad \modD_A = D_A - X_A \widetilde{W}\sharp\sharp.
\end{equation}
In particular, this definition covers the case $n=4$. The two definitions
of $\modDe$ and $\modD$ in \nn{modgen} and \nn{modeinst} for $n \not=4$
are consistent; the following
lemma follows from \cite[Section 4]{GoEinst}.

\begin{lemma}
On conformally Einstein manifolds  we have
$(n-4)\widetilde{W}_{ABCD} = W_{ABCD}$ for $n\not=4$. 
Therefore the operators $\modD$ and $\modDe$ defined in
\nn{modeinst} agree, in dimension $n\not=4$, with the operators
denoted by same symbols in expression \nn{modgen}.
\end{lemma}

To write explicitly the commutator $[\modD_A,\modD_B]$ on density valued
tractor fields, we shall need the following operator introduced in \cite{Goadv}.
Recall that sequentially labelled indices are assumed to be skew over,
e.g.\ $A^1A^2 = [A^1A^2]$. We put
\begin{equation} \label{doubleD}
D_{A^1A^2} = -2( w\W_{A^1A^2} + \X_{A^1A^2}^{\quad a} \na_a)
\end{equation}
Using this, one computes
\begin{gather} \label{[modD,modD]}
[\modD_A,\modD_B] =  (n+2w-2) \bigl[
(n+2w-4) \widetilde{W}_{AB} \sharp  - (D_{AB} \widetilde{W}) \sharp\sharp 
\bigr]. 
\end{gather}
on any density valued tractor field. Since we can use the definition \nn{modgen}
for $\modD$, in the case $n \not= 4$, the previous display follows from
\cite[(13)]{GoSiEHarm} for such dimensions. A direct computation then verifies
the case $n=4$.

\begin{lemma} \label{IIDD}
Let $I^A, \bar{I}^A \in \cE^A$ be two parallel tractors.  Then
$I^A \bar{I}^B[\modD_A,\modD_B]=0$ on any density valued tractor fields.
\end{lemma}

\begin{proof}
Since $I^AW_{ABCD}=0$, the case $n \not=4$ follows from 
\cite[(13) and Lemma 2.2(ii)]{GoSiEHarm}. Assume $n=4$. Then  
$I^A\widetilde{W}_{ABCD}=0$. Choosing an Einstein scale $\si$, this is easily 
verified using relations $\Om_{abCD}I^D=0$,
$\na_{[a}P_{b]c}=0$ and $\si^{-1}(\na_{[b}P_{a]d}) \na^a\si=0$, cf.\
the explicit formula of $\widetilde{W}_{ABCD}$ above. 
Further one easily verifies that \cite[Proposition 2.1 (ii)]{GoSiEHarm}
holds if we replace $W_{ABCD}$ by $\widetilde{W}_{ABCD}$. Therefore
also \cite[Lemma 2.2(ii)]{GoSiEHarm} holds if we replace $W_{ABCD}$ by 
$\widetilde{W}_{ABCD}$ and the case $n=4$ follows.
\end{proof}

\section{Einstein manifolds: conformal Laplacian operators on tractors}
\label{ops}

We assume henceforth that the structure $(M,[g])$ (is of dimension
$n\geq 3$ and) is conformally Einstein, and write $\si \in \cE[1]$ for
some Einstein scale from the conformal class. Then $I^A := \frac{1}{n}
D^A \si$ is parallel and $X^AI_A=\si$ is non-vanishing.

The operator $\modBox := \De + wJ + \widetilde{W} \sharp\sharp$ acting on 
tractor bundles of the weight
$w$ is conformally invariant only if $n+2w-2=0$. On the other hand the scale $\si$
(or equivalently $I^A$), yields the operator
\begin{equation} \label{modBox}
  \modBox_\si := I^A \modD_A = \si (-\modDe - 2\frac{w}{n}(n+w-1)J): 
  \cE_{B \cdots E}[w] \longrightarrow \cE_{B \cdots E}[w-1]
\end{equation}
which is well defined for any $w$, cf.\ \cite{GoEinst}. 
Thus we can consider
the composition $(\modBox_\si)^p$, $p \in \N$ and we set
$(\modBox_\si)^0 := \id$. These operators generally depend on the choice
of the scale $\si$ but one has the following modification of
\cite[Theorem 3.1]{GoEinst}.

\begin{theorem}\cite{GoSiKil} \label{indep}
Let $\si,\bar{\si}$ be two Einstein scales in the conformal
class and consider the operators 
$$ \frac{1}{\si^p} (\modBox_\si)^p, 
   \frac{1}{\bar{\si}^p} (\modBox_{\bar{\si}})^p:
   \cE_{B \cdots E}[w] \longrightarrow \cE_{B \cdots E}[w-2p],  $$
for $p\in \mathbb{Z}_{\geq 0}$.
If $w= p- n/2 $ then 
$\frac{1}{\si^p} (\modBox_\si)^p = 
\frac{1}{\bar{\si}^p} (\modBox_{\bar{\si}})^p$.
\end{theorem}

\begin{proof}
The proof is completely analogous to the proof of \cite[Theorem 3.1]{GoSiEHarm}
once we know $I^A \bar{I}^B[\modD_A,\modD_B]=0$. Hence the theorem follows using
Lemma \ref{IIDD}. 
\end{proof}

One can generalise the tractor--D operator to weighted forms as follows.
First some notation.
Exterior and interior multiplication by a tractor 1-form $\omega$ are given by
\begin{equation}\label{wedgie}
\begin{array}{rl}
(\varepsilon(\omega)\varphi)_{A_0\cdots A_k}&=(k+1)\omega_{[A_0}\varphi_{A_1\cdots A_k]}\,, \\
(\iota(\omega)\varphi)_{A_2\cdots A_k}&=\omega^{A_1}\varphi_{A_1\cdots A_k}\,.
\end{array}
\end{equation}
We extend the notation for interior and exterior multiplication in
an obvious way to operators which increase the rank by one.
For example, for $\varphi$ a weighted tractor form, $\io(\modD)\varphi$ means $\modD^{A_1}\varphi_{A_1\cdots A_k}$.
 
The operators 
$M_{\form{A}}^{\form{a}}[w]: \cE_{\form{a}} \to \cE_{\form{A}}[w-k]$ 
(see the operator $\overline{M}$ from \cite{GoSiKil})
and
$M^*{}^{\form{A}}_{\form{a}}: \cE_{\form{A}}[w'] \to \cE_{\form{a}}[w'+k]$
defined as
\begin{eqnarray} \label{M}
\begin{split}
  &M_{\form{A}}^{\form{a}} f_{\form{a}} =
     \frac{n+w-2k}{k} \Z_{\form{A}}^{\,\form{a}} f_{\form{a}}
     + \X_{\form{A}}^{\,\dform{a}} (\de f)_{\dform{a}}, 
   \quad f_{\form{a}} \in \cE_\form{a}[w] \\
  &M^*{}^{\form{A}}_{\form{a}} F_{\form{A}} =
     -(w'+k) \Z^{\form{A}}_{\,\form{a}} F_{\form{A}}
     + (d\X^{\form{A}} F_{\form{A}})_{\form{a}},
     \quad F_{\form{A}} \in \cE_\form{A}[w'].
\end{split}
\end{eqnarray}
where $\form{A} = \form{A}^k$ and $\form{a} = \form{a}^k$ are formal adjoints
for suitable choice of $w'$. (Hence $M^*$ is conformally invariant.)
These operators are closely related to $\io(\modD)\ep(X)$ and 
$\io(X)\ep(\modD)$. 
One easily computes that 
\begin{eqnarray*} 
\begin{split}
 &k(n+2(w-k)+2) M f =
     \io(\modD)\ep(X) q_k f, 
     \quad f \in \cE^k[w] \\
 &-(n+2w'-2) M^* F =
     q^k \io(X)\ep(\modD) F, 
     \quad F \in \cT^k[w'].
\end{split}
\end{eqnarray*}
It follows that on differential forms of generic weight 
\begin{equation}\label{dm}
\io(\modD)M=0.
\end{equation}
In fact by computing the remaining case it follows that this holds for
all weights.  Another consequence for $f$ and $F$ as above we have
\begin{eqnarray} \label{MM}
\begin{split}
  &\io(\modD)\ep(X) \io(X)\ep(\modD) F\!=\!
  -k(n\!+\!2w'\!+\!2)(n\!+\!2w'\!-\!2) MM^* F, \\
  &q^k\io(\!X\!)\ep(\modD)\io(\modD)\ep(\!X\!)q_k f \!=\! 
  -k(n\!+\!2(w\!-\!k)\!+\!2)(n\!+\!2(w\!-\!k)\!-\!2)M^*\!Mf, \\
  &\text{and}\ M^*Mf = -\frac{1}{k} w(n+w-2k) \id.
\end{split}
\end{eqnarray}
Note that contrary to the last relation, $MM^*$ is generally 
not a multiple of the identity.

Using these we obtain the (conformally invariant) operator 
\begin{eqnarray} \label{tildeD}
\begin{split}
  &\tmodD_B: f_{\form{a}}[w] \longrightarrow f_{\form{a}}[w-1]\\
  &\tmodD_B f_\form{a} =
   M^*{}^{\form{A}}_{\,\form{a}} \modD_B^{} M_\form{A}^\form{a} f_\form{a}.
\end{split}
\end{eqnarray}

As an analogy of \nn{modBox} we have (scale dependent) operators
\begin{equation} \label{tmodBox}
  \tmodBox^{(p)}_\si := 
  M^*{}^{\form{A}}_{\,\form{a}} (\modBox_\si)^p M_\form{A}^\form{a}: 
  \cE_\form{a}[w] \longrightarrow \cE_\form{a}[w-p]
\end{equation}
for $p \geq 1$ and we put $\tmodBox^{(0)}_\si := \id$. 
The case $p=1$ shall be denoted simply as
$\tmodBox_\si := \tmodBox^{(1)}_\si$. 
Note the operators $(\tmodBox_\si)^p$ and $\tmodBox^{(p)}_\si$ are 
generally different for $p\geq 2$. Using the formulae \nn{tractorD}, \nn{na} and \nn{dde-W} in a computation we obtain
\begin{eqnarray} \label{tmodbox1}
\begin{split}
  \tmodBox_\si^{(1)} = \tmodBox_\si = - \frac{1}{k} \si \bigl[
  &w(n+w-2k-1) d\de + (w-1)(n+w-2k) \de d \\
  &-2\frac{w(w-1)}{n} (n+w-2k) (n+w-2k-1) J \bigr].
\end{split}
\end{eqnarray}
and 
\begin{eqnarray*}
\begin{split}
  \tmodBox_\si^{(2)} &= 
   -\frac{1}{k} \si^2 \bigl[ w(n+w-2k-2) (d\de)^2 + (w-2)(n+w-2k) (\de d)^2 \\ 
  &-\frac{2}{n}w (n+w-2k-2) [ (w-1)(n+w-2k) + (w-2)(n+w-2k-1)] Jd\de \\
  &-\frac{2}{n}(w-2) (n+w-2k) [ (w-1)(n+w-2k-2) + w(n+w-2k-1)] J\de d \\
  &+\frac{4}{n^2} w(w-1)(w-2)(n+w-2k)(n+w-2k-1)(n+w-2k-2) J^2 \bigr].
\end{split}
\end{eqnarray*}
Actually one computes
$(\tmodBox_\si)^2 = -\frac{1}{k}(w-1)(n+w-2k-1) \tmodBox_\si^{(2)}$ from last 
two displays. This shows
that for $w \not\in \{1,-n+2k+1\}$, $\tmodBox_\si^{(2)}$ can be always 
decomposed into simpler 
factors. Later we shall need the case $w=1$ and remarkably, this can be also 
decomposed. On $\cE^k[1]$ we have
\begin{multline} \label{tmodbox2}
  \tmodBox_\si^{(2)} = -\frac{2}{k} \si^2 \Bigl[ 
  (\frac{n}{2}-k-\frac{1}{2})d\de + (\frac{n}{2}-k+\frac{1}{2})\de d \Bigr] 
  \Bigl[ d\de  - \de d 
  + \frac{4}{n} (\frac{n}{2}-k)J \Bigr]
\end{multline}
where the right hand side is not a scalar multiple of $(\tmodBox_\si)^2$.

\section{Conformal operators on weighted forms} \label{four}

As above we assume $(M,[g])$ is conformally Einstein and of any
dimension $n\geq 3$.  Henceforth we we also assume that $k$ is an integer
in the range $1 \leq k \leq \frac{n}{2}$.  As a further point of notation: we
shall write $\ph \sim \ps$ for differential operators $\ph$ and $\ps$ if
they are equal up to a nonzero scalar multiple. Further we assume $\si
\in \cE[1]$ is an Einstein scale and denote by $I^A$ the corresponding parallel tractor.  Our
main aim is to study the following class of operators:
  \begin{definition} \label{Ldef}
For each positive integer $\ell$, and $k$ as above, we shall define
the operator $L_k^\ell$, $\ell \geq 0$ by the formula
\begin{equation} \label{L}
  L_k^\ell := \si^{-\ell} q^k (\modBox_\si)^\ell M:
  \cE^k[w] \longrightarrow  \cE^k [w-2\ell],
  \quad w := k+\ell-n/2.
\end{equation}
\end{definition}

This operator has several important properties coming from its
relation to the operators denoted by the same symbol $L_k^\ell$ in
\cite{BrGodeRham}. These are defined there using the Fefferman-Graham
ambient metric of \cite{FGbook} and its link to the tractor connection
(as treated in \cite{CapGoamb,GoPetLap}).  First note the ambient
metric (denoted by $\h$ in \cite{GoSiEHarm}) exists and is unique (up
to an ambient diffeomorphism) to all orders in our case.  This is
always true in odd dimensions and, using \cite[Proposition
  7.5]{FGbook}, on conformally Einstein manifolds there is also a
canonical ambient metric to all orders in even dimensions. This in
particular means the operators $L_k^\ell$ in \cite{BrGodeRham} exist
and are unique for all $\ell \geq 1$, i.e.\ without any restrictions in
even dimensions.

Second, we recall that $L_k^\ell$ is defined in \cite{BrGodeRham} using
the ambient form Laplacian (denoted by $\afl$ in \cite{GoSiEHarm} and
\cite{BrGodeRham}).  The crucial fact for us is that the powers
$\afl^\ell$ can be rewritten for a chosen Einstein scale as in
\cite[Proposition 7.4]{GoSiEHarm}. Since the version of the ambient
$D$-operator (denoted by $\afD$ in \cite{GoSiEHarm})) descends to the
tractor operator $\modD$ in \nn{modgen} for $n \not=4$, we conclude
that the operator \nn{L} coincides with $L_k^\ell$ from
\cite{BrGodeRham} up to a scalar multiple for $n \not=4$. Here we used
$\io(X) (\modBox_\si)^\ell M =0$, for $w$ as in \nn{L}, which follows
from the discussion above (in particular from \nn{dm} and \cite[Lemma
  4.2]{BrGodeRham} with \cite[Proposition 7.4]{GoSiEHarm}).  Next
observe that if we calculate the right hand side of \nn{Wform4} in an
Einstein scale $\si$ then the terms involving the Cotton tensor terms
are zero, and we obtain the simplified expression
$\widetilde{W}_{ABCD} = \Z_{AB}^{\;a\;b} \Z_{CD}^{\;c\;d}
C_{abcd}$. Using this it follows easily that on conformally Einstein
manifolds $\widetilde{W}_{ABCD}$ coincides with the tractor field
determined by curvature of the ambient metric. Thus the ambient
$D$-operator $\afD$ descends to tractor operator $\modD$ of
\nn{modeinst} and it follows that the operator \nn{L} coincides with
$L_k^\ell$ from \cite{BrGodeRham} also in the case $n =4$.

Finally, observe that it follows from Theorem \ref{indep} that
$L_k^\ell$ is independent of the choice of Einstein scale $\si$, and
so this is another route to establishing uniqueness of these operators
on conformally Einstein manifolds.

\vspace{1ex}

It is also interesting to understand the meaning of  the ``bottom slot'' of 
$ \si^{-\ell} (\modBox_\si)^\ell Mf$ for $f \in \cE^k[w]$, i.e.\ the operator 
\begin{equation} \label{G}
  G_k^{\ell,\si} := \si^{-\ell} q^{k-1} \io(Y) (\modBox_\si)^\ell M:
  \cE^k[w] \longrightarrow  \cE^k [w\!-\!2\ell\!-\!2],
  \quad w := k\!+\!\ell\!-\!n/2,
\end{equation}
which depends on the choice of the Einstein scale $\si$.

The operators $L_k^1$ and $G_k^{1,\si}$ are particularly simple:

\begin{theorem} \label{L1}
Let $w=k+1-n/2$. The operators
$$ 
L_k^1: \cE^k[w] \to \cE^k[w-2] \quad \text{and} \quad
G_k^{1,\si}: \cE^k[w] \to \cE^{k-1}[w-4] 
$$
have explicit form
\begin{align*}
&L_k^1 = \bigl( \frac{n}{2}-k-1 \bigr) d\de + 
\bigl(\frac{n}{2}-k+1 \bigr) \de d
+ \frac{2}{n} \bigl(\frac{n}{2}-k-1 \bigr) \bigl(\frac{n}{2}-k+1 \bigr) 
\bigl( \frac{n}{2}-k \bigr) J \\
&G_k^{1,\si} = \de \bigl[ d\de 
+ \frac{2}{n} \bigl(\frac{n}{2}-k+1 \bigr) \bigl(\frac{n}{2}-k \bigr) J\bigr].
\end{align*}
\end{theorem}

\noindent
Note that up to a nonzero scalar multiple, $L_{n/2}^1$ simplifies to
$d\de-\de d$.  On the other hand, there is the relation $G_k^{1,\si} =
\frac{1}{n/2-k-1} \de L_k^1$ for $\frac{n}{2}-k-1 \not= 0$.  The
latter conditions excludes true forms, i.e.\ the weight $w=0$.

\begin{proof}
The theorem follows by a direct computation. Concerning $L_k^1$ for 
$k \not= \frac{n}{2}$ one can proceed also by the following (simpler) way. 
We use once again that \cite[Proposition 7.4]{GoSiEHarm}
 with \cite[Lemma 4.2]{BrGodeRham}
implies  that  $\io(X)\modBox_\si M=0$ on $\cE^k[w]$. 
Thus $M^*\modBox_\si M = -(w'+k)q^k\modBox_\si M$ on $\cE^k[w]$,
cf.\ \nn{M}, where $w'=w-1$. Therefore comparing 
$L_k^1 =  \si^{-\ell} q^k (\modBox_\si)^\ell M$ with 
$\frac{1}{\si} \widetilde{\modBox}_\si = \frac{1}{\si} M^* \modBox_\si M$, we 
see these coincide up to a constant multiple for $w'-1 \not =0$, i.e.\
for $k \not= \frac{n}{2}$. Thus the form of $L_k^1$ follows from \nn{tmodbox1}
in the latter case. 
\end{proof}

Before studying operators $L_k^\ell$ in detail, we describe their relation  
with the operators $G_k^{\ell,\si}$.

\begin{theorem} \label{LG}
Let $w=k+\ell-n/2$ where $1 \leq k \leq \frac{n}{2}$. Computing
in an Einstein scale $\si$, operators
$G_k^{\ell,\si}: \cE^k[k+l-n/2] \longrightarrow  \cE^{k-1} [k-\ell-n/2-2]$ 
satisfy 
\begin{align*}
&w G_k^{\ell,\si} = -\de L_k^\ell, \\
&G_k^{\ell,\si} = \frac{k-1}{k(n+w-2k+1)}\si^{-1} L_{k-1}^{\ell} \si\de
\quad \text{for} \quad k \geq 2.
\end{align*}
\end{theorem}

This result means that the operators $G_k^{\ell,\si}$ are not especially 
interesting for $w \not=0$. This is in strong contrast to the case
$w=0$, see \cite{GoSiEHarm} for details.  Note also the denominator in
the second display is always nonzero as $n+w-2k+1 =
n+(k+\ell-n/2)-2k+1 = (\frac{n}{2}-k) + (\ell+1) \geq 1$.

\begin{proof}
Since $\io(X) (\modBox_\si)^\ell M=0$ on $\cE^k[w]$ (as explained above), 
and using 
definitions of $L_k^\ell$ and $G_k^{\ell,\si}$ in \nn{L} and \nn{G},
these operators appear in the tractor field
\begin{equation} \label{sectF}
F_\form{A} := \si^{-\ell} \bigl( (\modBox_\si)^\ell M f \bigr)_\form{A} = 
\Z_\form{A}^{\,\form{a}} (L_k^\ell f)_\form{a}
+ k\X_\form{A}^{\,\dform{a}} (G_k^{\ell,\si}f)_{\dform{a}} \in \cE_\form{A}[-\frac{n}{2}-\ell]
\end{equation}
for $f \in \cE_\form{a}$. Here we use form abstract indices 
$\form{A} = \form{A}^k$ and $\form{a} = \form{a}^k$ as above. The weight
on the right hand side is obtained as $w-k-2\ell = -\frac{n}{2}-\ell$.
This tractor form has the property
$$
\io(\modD)F = \io(\modD) \si^{-\ell} (\modBox_\si)^\ell M f
= \si^{-\ell} (\modBox_\si)^{\ell+1} \io(X) M f =0,
$$
as follows from \cite[Proposition 7.4]{GoSiEHarm}
with \cite[Lemma 4.2]{BrGodeRham}. 
It remains to evaluate 
$0=(\io(\modD)F)_{\dform{A}} = -\modD^B F_{B\dform{A}}$ in detail. Using
the form of $F$ from \nn{sectF}, one gets
$$
\modD^B F_{B\dform{A}} = 2\ell \Z_{\dform{A}}^{\,\dform{a}} \bigl[
\bigl( k+\ell-\frac{n}{2} \bigr) (G_k^{\ell,\si}f)_{\dform{a}} + 
(\de L_k^\ell f)_{\dform{a}} \bigr]
+ \X_{\dform{A}}^{\,\ddform{a}} \nu_{\ddform{a}}
$$ for some section $\nu_{\ddform{a}} \in
\cE_{\ddform{a}}[-\frac{n}{2}-\ell+(k-4)]$.  Since $w =
k+\ell-\frac{n}{2}$ and the displayed tractor field is zero, the first
relation of the theorem follows.

To prove the second relation, observe that
$\io(Y) (\modBox_\si)^\ell M = \si^{-1} \io(I) (\modBox_\si)^\ell M  = 
\si^{-1} (\modBox_\si)^\ell \io(I)M$. Here and below we compute everything in the
scale $\si$. Since
$k\frac{n+w-2k+1}{k-1} \io(I)M = M\si\de$ using \nn{M}, where 
$\si$ denotes the multiplication by $\si$ on the right hand side,
and $G_k^{\ell,\si} = \si^{-\ell} q^{k-1} \io(Y) (\modBox_\si)^\ell M$ we obtain
$$
k\frac{n+w-2k+1}{k-1} G_k^{\ell,\si} = 
k\frac{n+w-2k+1}{k-1} \si^{-\ell} q^{k-1} \io(Y) (\modBox_\si)^\ell M =
\si^{-1} L_{k-1}^{\ell} \si\de.
$$
and the theorem follows.
\end{proof}

\section{Factorisation of operators $L_k^\ell$}

Our aim is to find explicit expressions for $L_k^\ell$ in the
general case $\ell \geq 1$. This also yields an explicit form for
$G_k^{\ell,\si}$ due to Theorem \ref{LG}. As above we assume the
manifold is conformally Einstein, $\si \in \cE[1]$ is an Einstein
scale, and  we write $I^A$ to denote the parallel
tractor corresponding to $\si$.  First we prove the following result.

\begin{lemma} Assume $V \in \cT^k[\ell_0-n/2]$ where $\ell_0 \in \{0,1,\cdots\}$ satisfies $\io(X)(\modBox_\si)^{\ell_0}V=0$. Then
$$ M^*M q^k \si^{-\ell_0} (\modBox_\si)^{\ell_0} V = q^k\si^{-\ell_0}
  (\modBox_\si)^{\ell_0} MM^* V. $$
\end{lemma}

\begin{proof} 
The case $\ell_0=0$ is an easy computation using \nn{M} and \nn{MM}.
We shall prove the case $\ell_0=1$ and $\ell_0 \geq 2$ separately.

(i) Assume $\ell_0=1$. We use the direct computation.
First we shall use the assumption on 
$$ 
V_\form{A} = \Y_\form{A}^{\,\dform{a}} \ka_{\dform{a}} +
\Z_\form{A}^{\,\form{a}} \mu_\form{A} + \W_\form{A}^{\,\ddform{a}}
\nu_{\ddform{a}} + \X_\form{A}^{\,\dform{a}} \rh_{\dform{a}} \in
\cE_\form{A}[1-n/2], $$ 
i.e.\ that $\io(X)(\modBox_\si)^{\ell_0}V =
\io(X) \si (-\modDe +\frac{n-2}{2}J)V =0$.  Here we have used the
formula \nn{modBox}. We shall need only the top slot of this tractor,
i.e.\
$$ \bigl[ \bigl( d\de + \de d + (1 - \frac{2(k-1)(n-k+1)}{n} +
  \frac{n-2}{2})J \bigl) \ka - 2k (\de \mu) + (n-2k+2) \rh +
  \frac{2}{k-1} d\nu \bigr]_{\dform{a}} $$ 
(using \nn{dde-W}); by our
assumption this is zero. Applying the differential $d$ to the last
display we obtain
\begin{equation} \label{topslot}
 \bigl[ \bigl( d\de + (\frac{n}{2} - \frac{2(k-1)(n-k+1)}{n})J \bigl) d\ka
   - 2k d\de \mu + 2(\frac{n}{2}-k+1) d\rh \bigr]_{\dform{a}} =0.
\end{equation}

To prove the Lemma we compare both sides of the claimed equality. 
To compute the right hand side we first need that
$$ (MM^* V)_\form{A} = 
   \frac{1}{k} (\frac{n}{2}-k+1) \Z_\form{A}^{\,\form{a}}
   [(\frac{n}{2}-k-1)\mu + \frac{1}{k}d\ka]_\form{a}
   + \X_\form{A}^{\,\dform{a}}
   [(\frac{n}{2}-k-1) \de\mu + \frac{1}{k} \de d\ka ]_{\dform{a}} $$
using \nn{M}. Applying $q^k\si^{-1} \modBox_\si = q^k(-\modDe +\frac{n+2}{2}J)$
to the last display we see that the right hand side 
$q^k\si^{-1} (\modBox_\si) MM^* V$ is equal to
\begin{align*}
& \frac{1}{k} (\frac{n}{2}-k+1) (\frac{n}{2}-k-1) \bigl[
  d\de + \de d - \frac{2k(n-k-1)}{n}J + \frac{n-2}{2}J \bigr] \mu + \\
& \frac{1}{k} (\frac{n}{2}-k-1) \bigl[ \frac{1}{k} d\de d\ka - 2d\de\mu \bigr]
  +\frac{1}{k^2} (\frac{n}{2}-k+1) \bigl[-\frac{2k(n-k-1)}{n}+\frac{n-2}{2}\bigr] 
  Jd\ka.
\end{align*}
using \nn{dde-W} after some computation.
The computation for the left hand side is simpler and we obtain that
$M^*M q^k\si^{-1} (\modBox_\si) V$ is equal to 
\begin{equation*}
  \frac{1}{k} (\frac{n}{2}-k+1) (\frac{n}{2}-k-1)\bigl[
  \bigl( d\de + \de d - \frac{2k(n-k-1)}{n}J + \frac{n-2}{2}J \bigr) \mu 
  - \frac{2}{nk} Jd\ka -\frac{2}{k} d\rh \bigr]_\form{a}
\end{equation*}
using \nn{dde-W} and \nn{MM}.
Now a short computation reveals that the difference of the last two displays
vanishes due to \nn{topslot}. 

(ii) Now assume $\ell_0 \geq 2$. 
Using once again \cite[Proposition 7.4]{GoSiEHarm}
with \cite[Lemma 4.2]{BrGodeRham} 
we have
$$ \io(X)\ep(\modD)\io(\modD)\ep(X) \si^{-\ell_0} (\modBox_\si)^{\ell_0}V =
   \si^{-\ell_0} (\modBox_\si)^{\ell_0} \ep(\modD)\ep(X)\io(X)\ep(\modD)V. $$
The left hand side is equal to 
$\io(X)\ep(\modD)\io(\modD)\ep(X)q_kq^k \si^{-\ell_0} (\modBox_\si)^{\ell_0}V$
due to the assumption $\io(X)(\modBox_\si)^{\ell_0}V=0$. 
Therefore applying $q^k$ to both sides of the previous display and using \nn{MM} 
we obtain
\begin{multline*}
  -4k(\ell_0-1)(\ell_0+1) M^*M q^k\si^{-\ell_0} (\modBox_\si)^{\ell_0} V
  = -4k(\ell_0+1)(\ell_0-1) q^k\si^{-\ell_0} (\modBox_\si)^{\ell_0} MM^* V
\end{multline*}
since $V \in \cT^k[\ell_0-n/2]$ and 
$q^k \si^{-\ell_0} (\modBox_\si)^{\ell_0}V \in \cE^k[k-\ell_0-n/2]$.
The scalar factor can be omitted on both sides
since we assume $k \geq 1$ and $\ell_0 \geq 2$, and the Lemma follows.
\end{proof}

The operator $L_k^\ell$ is
defined using a power of $\modBox_\si$. The following Theorem shows how to 
replace the factors $\modBox_\si$ (which act on tractor forms) by factors 
$\tmodBox_\si^{(p)}$ (which act on tensor forms).

\begin{theorem} \label{MM*}
Let $1 \leq p \leq \ell-1$. The operator
$L_k^\ell: \cE^k[k+l-n/2] \longrightarrow  \cE^k [k-\ell-n/2]$ 
satisfies
\begin{gather*}
\frac{1}{k} \bigl( k+(\ell-p)-n/2 \bigr) \bigr( k-(\ell-p)-n/2 \bigr)  
L_k^\ell = \si^{-p} L_k^{\ell-p} \tmodBox^{(p)}_\si.
\end{gather*}
\end{theorem}

\begin{proof}
Using \nn{L} we get 
$$ L_k^\ell := q^k \si^{-\ell} (\modBox_\si)^\ell M = 
   \si^{-p} q^k \si^{-(\ell-p)} (\modBox_\si)^{\ell-p} (\modBox_\si)^{p} M
$$
where, recall,  $\io(X)(\modBox_\si)^{\ell-p} (\modBox_\si)^{p} M =0$.
Using the previous Lemma (with $\ell_0 = \ell-p$) we have
$$ \si^{-p} M^*Mq^k \si^{-(\ell-p)} (\modBox_\si)^{\ell-p} (\modBox_\si)^{p} M
 = \si^{-p} q^k \si^{-(\ell-p)} (\modBox_\si)^{\ell-p} MM^*(\modBox_\si)^{p} M. 
$$
Now using \nn{MM} on the left hand side and \nn{tmodBox} on the right hand side
we finally obtain
$$ \frac{1}{k} \bigl( k+(\ell-p)-n/2 \bigr) \bigr( k-(\ell-p)-n/2 \bigr) 
   q^k \si^{-\ell} (\modBox_\si)^{\ell}  M
   = \si^{-p} q^k \si^{-(\ell-p)} (\modBox_\si)^{\ell-p} M 
   \tmodBox_\si^{(p)} $$ 
and the Theorem follows.
\end{proof}

The crucial point is that we can use the theorem repeatedly to
decompose $L_k^\ell$, $\ell \geq 2$ into a composition of factors
$\tmodBox^{(p)}_\si$.  For most weights involved we can apply the
theorem $(\ell-1)$ times with $p=1$. This yields a composition of
second order factors, each of which is $\tmodBox_\si =
\tmodBox_\si^{(1)}$ (although the weight of the form this is applied
to varies) apart from the left factor which is $L_k^1$.  We know
$L_k^1$ explicitly from Theorem \ref{L1}.  However the choice $p=1$ is
not available when this would imply that one of scalars on the left
hand side in Theorem \ref{MM*} is zero. In that case, we use the
choice $p=2$ which yields the factor $\tmodBox^{(2)}_\si$. For all
weights concerned we may decompose entirely using, at each step,
either $p=1$ or $p=2$.

\vspace{1ex}

The scalars on the left hand side in Theorem \ref{MM*} are
$k+(\ell-p)-n/2$ and $k-(\ell-p)-n/2$. The latter scalar is always
negative as $k \leq n/2$ and $\ell-p \geq 1$. Since we assume
$w=k+l-n/2$, the first scalar is equal to $w-p$ so the choice $p=1$
must be avoided only for the weight $w=1$. This affects only the even
dimensional case as $w=k+\ell-n/2$.  Moreover, for the special case
$k=\frac{n}{2}$ we have $w=1$ implies $\ell=1$ and this is known due
to Theorem \ref{L1}. Thus, in this process, when $w=1$ (then necessarily
$n$ is even) we are forced to use $p=2$ only if $k \leq n/2-1$.

We shall use the notation 
\begin{eqnarray} \label{Yam}
\begin{split}
  P^\Ph_k[E,F] := \prod_{i \in \Ph} \Bigl[ 
  & (w-i+1)(w-i+n-2k)E + (w-i)(w-i+n-2k+1)F \\
  & - \frac{2}{n} (w-i)(w-i+1)(w-i+n-2k)(w-i+n-2k+1)J \Bigr] 
\end{split}
\end{eqnarray}
where $\Ph \subseteq \Z$ is a finite set and $E,F: \cE^k[\bar{w}] \to
\cE^k[\bar{w}]$, $\bar{w} \in \R$ are differential operators. In fact
we use this with $E: =d\de$ and $F := \de d$. Then note that each
factor on the right hand side, with fixed $i \in \Ph$, is (up to the
multiple $-\frac{1}{k}\si$) just $\tmodBox_\si$ from 
\nn{tmodbox1} where  $w$ is replaced by $w-i+1$.

As mentioned above, the only case we need the choice $p=2$ in Theorem
\ref{MM*} is on $\cE^k[1]$, $k \leq \frac{n}{2}-1$. This leads to the
factor $\tmodBox_\si^{(2)}$; but this is decomposed in
\nn{tmodbox2}. From this it follows that $\tmodBox_\si^{(2)}$ is equal
(up to a nonzero scalar multiple) to
\begin{equation} \label{SqYam}
  \widetilde{P}_k(E,F) = \Bigl[ 
  (\frac{n}{2}-k-\frac{1}{2})E + (\frac{n}{2}-k+\frac{1}{2})F \Bigr]
  \Bigl[ E-F + \frac{4}{n}(\frac{n}{2}-k)J \Bigr].
\end{equation}
where $E=d\de$ and $F=\de d$.

With this notation we are ready to state the main factorisation result:

\begin{theorem} \label{Lfactor}
Let $\Ph := \{1,\ldots,\ell\}$ and $w=k+l-n/2$ where $1 \leq k \leq \frac{n}{2}$. 
The operator
$$ 
L_k^\ell: \cE^k[w] \to \cE^k[w-2\ell]
$$
has the explicit form
$$
L_k^\ell \sim \begin{cases}
\si^{-\ell} (\tmodBox_\si)^\ell \sim  
(d\de-\de d) P_k^{\Ph \setminus \{\ell\}}(d\de,\de d)
& k=n/2 \\
\si^{-\ell} (\tmodBox_\si)^\ell \sim P_k^{\Ph}(d\de,\de d)
& w \leq 0 \wedge k<n/2 \\
\si^{-\ell} (\tmodBox_\si)^{\ell-w-1} \tmodBox_\si^{(2)} (\tmodBox_\si)^{w-1}
\sim \widetilde{P}_k(d\de,\de d)P_k^{\Ph \setminus \{w,w+1\}}(d\de,\de d)
& w \geq 1 \wedge k<n/2,
\end{cases} 
$$
for $n$ even and 
$$
L_k^\ell \sim \si^{-\ell} (\tmodBox_\si)^\ell \sim P_k^{\Ph}(d\de,\de d)
$$
for $n$ odd. Here $\sim$ means ``is equal up to nonzero scalar multiple'' and 
we use the notation from \nn{Yam} and \nn{SqYam}.
\end{theorem}

\begin{proof}
In the case $k=\frac{n}{2}$ we first use Theorem \ref{MM*} with $p=1$
$(\ell-1)$-times. This yields $P_k^{\Ph \setminus \{\ell\}}(d\de,\de
d)$. Then remains, as the leftmost factor, just $L_k^1$. This is equal
to $d\de-\de d$ (up to a nonzero scalar) according to Theorem
\ref{L1}.  The case $w \leq 0$, $k<n/2$ for $n$ even and  the case
$n$ is odd are even simpler, we simply use  Theorem
\ref{MM*} $\ell$-times with $p=1$.

Now assume  $w \geq 1$ and $ k<n/2$. First we apply Theorem \ref{MM*}
$(w-1)$ times with $p=1$; this yields the factor 
$P_k^{\{1,\ldots,w-1\}}(d\de,\de d)$.
Then we apply Theorem \ref{MM*} with $p=2$ which introduces the factor
$\widetilde{P}_k(d\de,\de d)$. Then we continue with an iterated use
of Theorem \ref{MM*} with $p=1$, and this  yields
$P_k^{\{w+2,\ldots,\ell\}}(d\de,\de d)$. From these three steps, the
theorem follows.
\end{proof}

Note the previous theorem provides also factorisation of the operators
$G_k^{\ell,\si}$, due to Theorem \ref{LG}.

\begin{remark}
The explicit formulae for $L_k^\ell$ show that these operators
are formally self-adjoint. Indeed, $L_k^\ell$ is a polynomial in $E=d\de$
and $F=\de d$ where actually only monomials $E^p$ and $F^q$ appear. 
Thus $F^*=F$ and $E^*=E$ immediately implies that $(L_k^\ell)^* = L_k^\ell$.
\end{remark}

\section{Decomposition of the null space of $L_k^\ell$.}

We continue the notation and setting from the previous
section. Henceforth shall use notation $\cN(F)$ and $\cR(F)$ for the
null space and range, respectively of an operator $F$.  Recall every
operator $L_k^\ell$, $1 \leq k \leq \frac{n}{2}$ is a composition of
$\ell$ second order commuting factors, each of them of the form $ad\de
+ b\de d +c$ where $a,b,c$ are (half) integers. In this Section we use
our results above to produce a direct sum decomposition of the null
space of the operators $L_k^\ell$.

\subsection{Riemannian signature and $M$ closed}
Assuming the Einstein metric
has  Riemannian signature and that $M$ is closed (i.e.\ compact,
without boundary), we obtain easily an explicit description of
$\cN(L_k^\ell)$.  Recall the space of $k$-forms decomposes as
$$
\cE^k = \cR(d) \oplus \cR(\de) \oplus (\cN(d) \cap \cN(\de))
$$
where $\cN(d) \cap \cN(\de)$ is the space of harmonic forms. Both $\cR(d)$ and 
$\cR(\de)$ decompose to eigenspaces of the form Laplacian and using the notation
\begin{eqnarray*}
&\overline{\cH}^k_{\si,\la} := \{ f \in \cE^k \mid d\de f = \la f\}
\subseteq \cR(d), \quad
\widetilde{\cH}^k_{\si,\la} := \{ f \in \cE^k \mid \de d f = \la f\}
\subseteq \cR(\de) \\
&\text{and} \quad \cH^k_{\si} := \cN(d) \cap \cN(\de)
\end{eqnarray*}
from \cite{GoSiEHarm}, we have
\begin{equation} \label{2nd}
\cN(ad\de + b\de d +c) = \overline{\cH}^k_{\si,-\frac{c}{a}} \oplus 
\widetilde{\cH}^k_{\si,-\frac{c}{b}} 
\quad \text{and} \quad \cN(ad\de + b\de d) = \cH^k_{\si}
\end{equation}
for $a,b,c$ nonzero. Using the spectral decomposition given by the
form Laplacian with our result Theorem \ref{Lfactor} we immediately
obtain the following.

\begin{theorem} \label{LdecSpec}
Let $M$ be a closed manifold equipped with an Riemannian Einstein
metric which is not Ricci flat. Let $w=k+l-n/2$ where $1 \leq k \leq
\frac{n}{2}$ and put
$$
\bar{\la}_i = \frac{2}{n}(w-i)(w-i+n-2k+1)J \quad \text{and} \quad
\tilde{\la}_i = \frac{2}{n}(w-i+1)(w-i+n-2k)J,
$$
for $1 \leq i \leq \ell$, and $\mu = \frac{4}{n}(\frac{n}{2}-k)J$. 
Then
$$
\cN(L_k^\ell)= \begin{cases}
\cH^k_\si \oplus \bigoplus_{i=1}^{\ell-1} \Bigl(  
\overline{\cH}^k_{\si,\bar{\la}_i} \oplus \widetilde{\cH}^k_{\si,\tilde{\la}_i}
\Bigr) 
& k=n/2 \\
\bigoplus_{i=1}^{\ell} \Bigl(  
\overline{\cH}^k_{\si,\bar{\la}_i} \oplus \widetilde{\cH}^k_{\si,\tilde{\la}_i}
\Bigr) 
& w < 0 \wedge k<n/2 \\
\widetilde{\cH}^k_{\si,0} \oplus \bigoplus_{i=2}^{\ell} \Bigl(  
\overline{\cH}^k_{\si,\bar{\la}_i} \oplus \widetilde{\cH}^k_{\si,\tilde{\la}_i}
\Bigr) 
& w=0 \\
\cH^k_\si \oplus \overline{\cH}^k_{\si,\mu} \oplus \widetilde{\cH}^k_{\si,-\mu} 
\oplus \!\!
\bigoplus\limits_{\substack{i=1,\ldots,\ell \\ i \not\in \{w,w+1\}}} \Bigl(  
\overline{\cH}^k_{\si,\bar{\la}_i} \oplus \widetilde{\cH}^k_{\si,\tilde{\la}_i}
\Bigr) 
& w \geq 1 \wedge k<n/2,
\end{cases} 
$$
\end{theorem}

\subsection{The general case}

On general Einstein manifolds we do not have the spectral decomposition but
still can reduce the description of the null space of $L_k^\ell$ to a second 
order problem. First we show the following:
 
\begin{proposition} \label{identity}
Let $L_k^\ell = S_1 \ldots S_\ell: \cE^k \to \cE^k$ be the decomposition of 
$L_k^\ell$ to second 
order factors from Theorem \ref{Lfactor} in the Einstein metric $\si$ and 
assume $J \not=0$ in this scale. Choose 
a pair of integers $1 \leq t <u \leq \ell$. There are differential operators 
$\ph_t, \ph_u: \cE^k \to \cE^k$ such that
$$
\ph_t \circ S_t + \ph_u \circ S_u = \id 
$$
where the operators $\ph_t$, $\ph_u$, $S_t$ and $S_u$ mutually commute.
\end{proposition}

\begin{proof}
Considering the right hand side of operators $L_k^\ell$ in Theorem
\ref{Lfactor}, there are several possibilities for second order
factors of $S_t$ and $S_u$. First, they can come from the polynomial
$P_k^\Ph$. In this case the right hand side of \nn{Yam}, with
$w=k+\ell-\frac{n}{2}$, yields (up to a sign) factors
\begin{eqnarray}
\begin{split} 
&\widetilde{S}_i := (A_-+i-1)(A_+-i)E + (A_-+i)(A_+-i+1)F+ \\
&\qquad\qquad\qquad + \frac{2}{n} (A_-+i-1)(A_+-i)(A_-+i)(A_+-i+1) J, \\
&\text{where} \quad A_+ = \frac{n}{2}-k+\ell, \quad A_- = \frac{n}{2}-k-\ell, 
\quad 1 \leq i \leq \ell,\ i \not \in \{w,w+1\}.  
\end{split}
\end{eqnarray}
(Note the condition $i \not \in \{w,w+1\}$ is vacuous for $w \leq 0$.)
We shall call these factors \idx{generic}. Here and below we use the
notation $E = d\de$ and $F = \de d$, as in the previous Section. Next
if not as just described the remaining possible factors are
\begin{align*}
&\widetilde{S} := \bigl( \frac{n}{2}-k-\frac12 \bigr) E
+ \bigl( \frac{n}{2}-k+\frac12 \bigr) F, \\
&\widetilde{S}' := E-F +\frac4n \bigl( \frac{n}{2} -k \bigr) J
\qquad \text {and} \\
&\widetilde{S}'' = E-F.
\end{align*}
Factors $\widetilde{S}$ and $\widetilde{S}'$ come from the polynomial 
$\widetilde{P}_k$ in Theorem \ref{Lfactor}, i.e.\ from \nn{SqYam}, 
the factor $\widetilde{S}''$ appears in Theorem \ref{Lfactor} for 
$k=\frac{n}{2}$. Summarising, every factor of $S_t$ and $S_u$ from the 
Proposition \ref{identity} is either $\widetilde{S}_i$, $1 \leq i \leq \ell$ (generic
factors) or $\widetilde{S}$ or $\widetilde{S}'$ or $\widetilde{S}''$.

First consider a generic pair of factors $S_t$, $S_u$ i.e.\ 
$S_t =\widetilde{S}_i$, $S_u =\widetilde{S}_{i+p}$. That is,
\begin{align*}
S_t =& B_-(B_+-1)E + (B_-+1)B_+F +  \frac2n B_-(B_+-1)(B_-+1)B_+ J, \\
S_u =& (B_-+p)(B_+-p-1)E + (B_-+p+1)(B_+-p)F \\ 
&+ \frac2n (B_-+p)(B_+-p-1)(B_-+p+1)(B_+-p)J. 
\end{align*}
Here $B_+ = \frac{n}{2}-k+\ell-r$ and $B_- = \frac{n}{2}-k-\ell+r$ where
$r=i-1$ i.e.\ $0 \leq r \leq \ell-1$ and $0<p$ such that $r+p \leq \ell-1$. 
The conditions $i \not\in \{w,w+1\}$ and $i+p \not\in \{w,w+1\}$ then mean, 
respectively, $r \not\in \{w-1,w\}$ and $r+p \not\in \{w-1,w\}$.
Now a short computation reveals that 
$$
\overline{S}_u := \frac{1}{p(B_+ - B_- -(p+1))}  \bigl( S_u-S_t \bigr)
= E+F + \frac{2}{n} \bigl[ 2B_+B_- + (p+1)(B_+-B_--p) \bigr]J.
$$
where the denominator is nonzero since  
$B_+ - B_- -(p+1) = 2(\ell-r) -(p+1) = (\ell-r-p-1) + (\ell-r) >0$. 
It is sufficient to replace the pair of operators $S_t$, $S_u$ by the pair
$S_t$, $\overline{S}_u$ i.e.\ to find operators $\ph_t$ and $\ph_u$ such that
$\ph_t \circ S_t + \ph_u \circ \overline{S}_u = \id$. We shall obtain these 
operators in the form
\begin{equation} \label{scalars}
\ph_t = x_1 E + y_1 F + z_1 \quad \text{and} \quad
\ph_u = x_2 E + y_2 F + z_2
\end{equation}
where $x_1, y_1, z_1, x_2, y_2, z_2 \in \R$. We will use the notation
\begin{align*}
&\al_1 = B_-(B_+-1) \not=0, \quad \al_2 = B_+(B_-+1) \not=0 \quad\
\text{and}  \\
&\be = \frac{2}{n} \bigl[ 2B_+B_- + (p+1)(B_+-B_--p) \bigr].
\end{align*}
The inequalities follow from the following: $B_+>0$ always and 
$B_+= (\frac{n}{2}-k) + (\ell-r)=1$ only for $k = \frac{n}{2}$ and 
$\ell = r+1$ and hence  $w = k+\ell-\frac{n}{2}=r+1$. Next $B_-=0$ is equivalent 
to $r=k+\ell-\frac{n}{2}=w$ and
$B_-=-1$ is equivalent to $r=k+\ell-\frac{n}{2}-1=w-1$. (Recall
we assume $r \not\in \{w-1,w\}$.)

It remains to determine the scalars $x_1, y_1, z_1, x_2, y_2, z_2$. Putting
$x_2 = -\al_1 x_1$, $y_2 = -\al_2 y_1$ and $z_2=0$ we obtain
\begin{align*}
&\ph_t \circ S_t + \ph_u \circ \overline{S}_u = \\
& \quad 
= (x_1 E + y_1 F + z_1) \circ (\al_1 E + \al_2 F + \frac2n \al_1 \al_2 J) + 
(-\al_1 x_1 E - \al_2 y_1 F) \circ (E + F + \be J) = \\
&\quad = 
\bigl[ \bigl( \frac2n \al_1\al_2 - \be\al_1 \bigr) J x_1 +\al_1z_1\bigr]E+
\bigl[ \bigl( \frac2n \al_1\al_2 - \be\al_2 \bigr) J y_1 +\al_2z_1\bigr]F
+ \frac2n \al_1\al_2z_1J \id.
\end{align*}
We are looking for $x_1$, $y_1$, $z_1$ and $z_2$ such that only the
last term of the previous display is nonzero (and equal to the
identity).  Thus $z_1 := \frac{n}{2\al_1\al_2 J}$ and this determines
$x_1$ and $y_1$ provided
$$
\al_1 \bigl[ \frac2n \al_2 - \be \bigr] \not=0 \quad \text{and} \quad
\al_2 \bigl[ \frac2n \al_1 - \be \bigr] \not=0 .
$$ 
 We shall now verify the previous two inequalities. First we consider
$$
\al_1 \bigl[ \frac2n \al_2 - \be \bigr] = 
- \frac2n \al_1 \bigl[  B_+B_- + (p+1)(B_+-B_--p) - B_+ \bigr]
$$ 
where $\al_1 \not=0$. Observe that the second factor on the right hand
side is equal to zero only for $B_-=-p$ or $B_+ = p+1$. (Considering
the second factor as the quadratic expression in $p$, these are the two
roots.) The case $B_+ -p-1 = (\frac{n}{2}-k) + (\ell-r-p-1) =0$ is
equivalent to $\frac{n}{2}=k$ and $\ell = r+p+1$ i.e.\ $w =
k+\ell-\frac{n}{2} = r+p+1$ and the case $B_-=-p$ is equivalent to $-w
= \frac{n}{2}-k-\ell = -r-p$.  Since we assume $r+p \not\in
\{w-1,w\}$, the previous display is indeed nonzero. Second we consider
$$
\al_2 \bigl[ \frac2n \al_1 - \be \bigr] = 
- \frac2n \al_2 \bigl[  B_+B_- + (p+1)(B_+-B_--p) + B_- \bigr]
$$ 
where $\al_2 \not=0$. Observe now that the second factor on the
right hand side is equal to zero only for $B_-=-p-1$ or $B_+ = p$. 
Since $B_+ -p = (\frac{n}{2}-k) + (\ell-r-p) >0$
and $B_-=-p$ is equivalent to $-w+1 = \frac{n}{2}-k-\ell+1 = -r-p$
(and we assume $r+p \not\in \{w-1,w\}$), the previous display is also
nonzero.  Summarising, we have proved the proposition for generic
factors $S_t$, $S_u$.  In particular, we have proved the proposition
for $n$ odd.

Henceforth we assume $n$ even.
The remaining possibilities for how to chose the pair of operators $S_t$, $S_u$ are:
\begin{align*}
&(a)\ S_t \ \text{generic and\ } S_u = \widetilde{S} \ 
\text{for} \ k<\frac{n}{2}, \notag \\
&(b)\ S_t \ \text{generic and\ } S_u = \widetilde{S}' \ 
\text{for} \ k<\frac{n}{2}, \notag \\
&(c)\ S_t = \widetilde{S} \ \text{and\ } S_u = \widetilde{S}' \
\text{for} \ k<\frac{n}{2}, \notag \\
&(d)\ S_t \ \text{generic and\ } S_u = \widetilde{S}'' \
\text{for} \ k=\frac{n}{2}. \notag
\end{align*}
We shall discuss cases (a), (b), (c) first and 
assume $k<\frac{n}{2}$. Then we will treat (d) separately.
In the case (a) we deal with $S_t = \al_1 E + \al_2 F + \frac2n \al_1\al_2J$
and $S_u = aE + bF$ where $a = \frac{n}{2}-k-\frac12 \not=0$ and 
$b = \frac{n}{2}-k+\frac12 \not=0$ (since $n$ is even). We  find 
operators $\ph_u$ and $\ph_t$ in the form \nn{scalars} where 
$x_2 = - \al_1 \frac{x_1}{a}$, $y_2 = - \al_1 \frac{y_1}{b}$ and $z_2=0$. Then
\begin{align*}
&\ph_t \circ S_t + \ph_u \circ S_u = \\
& \quad 
= (x_1 E + y_1 F + z_1) \circ (\al_1 E + \al_2 F + \frac2n \al_1 \al_2 J) + 
(- \al_1 \frac{x_1}{a} E - \al_1 \frac{y_1}{b} F) \circ (aE + bF) = \\
&\quad = 
\bigl[ \frac2n \al_1\al_2J x_1 + \al_1z_1 \bigr] E+
\bigl[ \frac2n \al_1\al_2J y_1 + \al_2z_1 \bigr] F
+ \frac2n \al_1\al_2z_1J \id.
\end{align*}
We need the previous display to be equal to identity and since $\al_1 \not=0$
and $\al_2 \not=0$, this can be obviously achieved for unique $z_1$, $x_1$
and $y_1$.

In the case (b) we have $S_t = \al_1 E + \al_2 F + \frac2n \al_1\al_2J$
and $S_u = E-F + \frac4n (\frac{n}{2}-k)J$. In this case we find 
operators $\ph_u$ and $\ph_t$ in the form \nn{scalars} where 
$x_2 = - \al_1 x_1$, $y_2 = \al_2 y_1$ and $z_2=0$. Then putting
$\ga = 2 (\frac{n}{2}-k) \not=0$ we obtain
\begin{align*}
&\ph_t \circ S_t + \ph_u \circ S_u = \\
&\quad 
= (x_1 E + y_1 F + z_1) \circ (\al_1 E + \al_2 F + \frac2n \al_1 \al_2 J) 
+ (- \al_1 x_1 E + \al_2 y_1 F) \circ ( E-F + \frac2n \ga J) = \\
&\quad = 
\bigl[ \frac2n \al_1 \bigl( \al_2 - \ga \bigr) Jx_1 + \al_1z_1 \bigr] E 
+\bigl[ \frac2n \al_2 \bigl( \al_1 + \ga \bigr) Jy_1 + \al_1z_1 \bigr] F
+ \frac2n \al_1\al_2z_1J \id.
\end{align*}
We need the previous display to be equal to identity and this (uniquely) 
determines $z_1$, $x_1$ and $y_1$ provided $\al_2 - \ga \not=0$ and
$\al_1 + \ga \not=0$ where, observe, $\ga = B_+ + B_-$. Thus
$\al_2 - \ga = B_+(B_- +1) -(B_+ + B_-) = B_-(B_+ -1)$
and $\al_1 + \ga = B_-(B_+ -1) + (B_+ + B_-) = B_+(B_- +1)$. Since
$B_-\not=0$, $B_+ \not=0$ and we observed above that both $B_+=1$
$B_-=-1$ imply $r=w-1$ (recall we assume $r \not\in \{w-1,w\}$), we have indeed 
verified that $\al_2 - \ga \not=0$ and $\al_1 + \ga \not=0$.

In the case (c) we have $S_t = a E + b F$
and $S_u = E-F + \frac4n (\frac{n}{2}-k)J$ where 
$a = \frac{n}{2}-k-\frac12 \not=0$ and $b = \frac{n}{2}-k+\frac12 \not=0$. 
Here we find operators $\ph_u$ and $\ph_t$ in the form \nn{scalars} where 
$x_2 = - a x_1$, $y_2 = b y_1$ and $z_1=0$. Then putting
$\ga = 2 (\frac{n}{2}-k) \not=0$ we obtain
\begin{align*}
&\ph_t \circ S_t + \ph_u \circ S_u = \\
&\quad 
= (x_1 E + y_1 F) \circ (a E + b F) 
+ (- a x_1 E + b y_1 F + z_2) \circ ( E-F + \frac2n \ga J) = \\
&\quad = 
\bigl[ -\frac2n a\ga J x_1 + z_2 \bigr] E 
+\bigl[ \frac2n b\ga J y_1 - z_2 \bigr] F + \frac2n \ga z_2 J \id.
\end{align*}
One can obviously choose (uniquely) $x_1$, $y_1$ and $z_2$ such that
the previous display is equal to the identity.

Finally we shall discuss the case (d). We thus assume $k = \frac{n}{2}$.
Then $B_+ = \ell-r$ and $B_- = -(\ell-r)$ hence
$\al_1 = \al_2 = -(\ell-r)(\ell-r-1) =: \bar{\al} \not=0$.
Here the inequality follows since $\ell-r>0$ and $\ell-r-1=0$ would imply
$w = k+\ell-\frac{n}{2} = r+1$ (recall we assume $r \not\in \{w-1,w\}$).
Thus we can replace $S_t = \bar{\al}^2 ( E + F + \frac2n J)$
by $\overline{S}_t := E + F + \frac2n J$. The other factor is $S_u = E-F$.
We find operators $\ph_u$ and $\ph_t$ in the form \nn{scalars} where 
$x_2 = x_1$, $y_2 = y_1$ and $z_2=0$. Then
\begin{align*}
&\ph_t \circ \overline{S}_t + \ph_u \circ S_u = \\
&\quad 
= (x_1 E + y_1 F + z_1) \circ (E + F + \frac2n J) 
+ (- x_1 E + y_1 F) \circ ( E-F) = \\
&\quad = 
\bigl[ \frac2n J x_1 + z_1 \bigr] E 
+\bigl[ \frac2n J y_1 + z_1 \bigr] F + \frac2n J z_1 \id.
\end{align*}
It is obvious that there is a unique choice for $x_1$, $y_1$ and $z_1$
such that the previous display is equal to the identity. 
\end{proof}

Combining the previous proposition with \cite[Corollary 2.3]{GoSiDec}, we 
obtain the final result.

\begin{theorem} \label{LdecGen}
Let $L_k^\ell = S_1 \ldots S_\ell: \cE^k \to \cE^k$ be the decomposition of 
$L_k^\ell$ to second 
order factors from Theorem \ref{Lfactor} in the Einstein scale $\si$ and 
assume $J \not=0$ in this scale. Then
$$
\cN(L_k^\ell) = \cN(S_1) \oplus \ldots \oplus \cN(S_\ell).
$$
\end{theorem}

\end{document}